\documentclass[paper=a4, fontsize=12pt, abstract=on]{article}

\usepackage{amsmath}
\usepackage{amsthm}
\usepackage{amsfonts}
\usepackage{amssymb}
\usepackage{mathtools}
\usepackage{hyperref}
\usepackage{color}
\usepackage{xcolor}
\usepackage{dsfont}
\usepackage{mathrsfs}

\numberwithin{equation}{section}

\usepackage{titling}

\usepackage[headsepline]{scrlayer-scrpage}
\usepackage{geometry}

\usepackage[utf8]{inputenc}

\usepackage{lmodern}
\usepackage[T1]{fontenc}

\title{\bf Malliavin-Stein Method: \\ a Survey of Some Recent Developments   }
\author{Ehsan Azmoodeh\thanks{Department of Mathematical Sciences, University of Liverpool, E-mail: ehsan.azmoodeh@liverpool.ac.uk},
 	Giovanni Peccati\thanks{DMATH, Universit\'e du Luxembourg. E-mail: giovanni.peccati@gmail.com}
and Xiaochuan Yang\thanks{Department of Mathematical Sciences, University of Bath, E-mail: xiaochuan.j.yang@gmail.com} 	}
\date{\today}
\theoremstyle{plain}
\newtheorem{Thm}{Theorem}[section]
\newtheorem{Lem}[Thm]{Lemma}
\newtheorem{lem}[Thm]{Lemma}
\newtheorem{Prop}[Thm]{Proposition}
\newtheorem{prop}[Thm]{Proposition}

\newtheorem{Cor}[Thm]{Corollary}
\newtheorem{Def}[Thm]{Definition}
\newtheorem{Rem}[Thm]{Remark}

\newtheorem{Ex}[Thm]{Example}

\newtheorem{con}[Thm]{Conjecture}

\newcommand{\dom}{{\rm dom}\, }

\newcommand{\mfk}{\mathfrak{f}}

\providecommand{\Enorm}[1]{\lVert #1\rVert_2}

\newcommand{\diff}[1]{\operatorname{d}\ifthenelse{\equal{#1}{}}{\,}{\!#1}}
\newcommand{\me}{\ensuremath{\mathrm{e}}}

\newcommand{\HH}{\mathfrak{H}}

\def\E{\mathbb{E}}
\def \P{\mathbb{P}}
\def\R{\mathbb{R}}

\def\N{\mathbb{N}}

\newcommand{\K}{\mathrm{\textbf{Ker}}}

\newcommand{\Id}{\mathrm{\textbf{Id}}}
\newcommand{\LL}{{\bf L}}
\newcommand{\ud}{\mathrm{d}}

\newcommand{\xy}[1]{{\color{purple}{{#1}}}}

\DeclareMathOperator{\W}{W}
\DeclareMathOperator{\III}{I}
\DeclareMathOperator{\SSS}{S}
\DeclareMathOperator{\HHH}{H}
\DeclareMathOperator{\HS}{HS}
\DeclareMathOperator{\Hess}{Hess}
\DeclareMathOperator{\Ent}{Ent}
\DeclareMathOperator{\HSI}{HSI}
\DeclareMathOperator{\WSH}{WSH}

\DeclareMathOperator{\supp}{supp}

\DeclareMathOperator{\Var}{Var}

\DeclareMathOperator{\Cov}{Cov}

\DeclarePairedDelimiter\abs{\lvert}{\rvert}

\let\temp\epsilon \let\epsilon\varepsilon \let\varepsilon\temp
\let\temp\phi \let\phi\varphi \let\varphi\temp

\begin{document}
	\maketitle
	\begin{abstract} Initiated around the year 2007, the Malliavin-Stein approach to probabilistic approximations combines Stein's method with infinite-dimensional integration by parts formulae based on the use of Malliavin-type operators. In the last decade, Malliavin-Stein techniques have allowed researchers to establish new quantitative limit theorems in a variety of domains of theoretical and applied stochastic analysis. The aim of this survey is to illustrate some of the latest developments of the Malliavin-Stein method, with specific emphasis on extensions and generalisations in the framework of Markov semigroups and of random point measures.
	  
	\end{abstract}



	\tableofcontents
\section{Introduction and overview}\label{s:intro}

The {\bf Malliavin-Stein method} for probabilistic approximations was initiated in the paper \cite{StMethOnWienChaos}, with the aim of providing a quantitative counterpart to the (one- and multi-dimensional) central limit theorems for random variables living in the Wiener chaos of a general separable Gaussian field. As formally discussed in the sections to follow, the basic idea of the approach initiated in \cite{StMethOnWienChaos} is that, in order to assess the discrepancy between some target law (Normal or Gamma, for instance), and the distribution of a non-linear functional of a Gaussian field, one can fruitfully apply infinite-dimensional integration by parts formulae from the {\bf Malliavin calculus of variations} \cite{Mbook, n-p-book, GelbesBuch, Nbook} to the general bounds associated with the so-called {\bf Stein's method} for probabilistic approximations \cite{n-p-book, Chen-book}. In particular, the Malliavin-Stein approach captures and amplifies the essence of \cite{cha}, where Stein's method was combined with finite-dimensional integration by parts formulae for Gaussian vectors, in order to deduce {\bf second order Poincar\'e inequalities} -- as applied to random matrix models with Gaussian-subordinated entries (see also \cite{NPR09, v}).

\medskip

We recall that, as initiated by P. Malliavin in the path-breaking reference \cite{Mpaper}, the Malliavin calculus is an infinite-dimensional differential calculus, whose operators act on smooth non-linear functionals of Gaussian fields (or of more general probabilistic objects). As vividly described in the classical references \cite{Mbook, GelbesBuch}, as well as in the more recent books \cite{n-p-book, Nbook}, since its inception such a theory has generated a staggering number of applications, ranging e.g. from mathematical physics to stochastic differential equations, and from mathematical finance to stochastic geometry, analysis on manifolds and mathematical statistics. On the other hand, the similarly successful and popular Stein's method (as created by Ch. Stein in the classical reference \cite{Stein72} -- see also the 1986 monograph \cite{S86}) is a collection of analytical techniques, allowing one to estimate the distance between the distributions of two random objects, by using characterising differential operators. The discovery in \cite{StMethOnWienChaos} that the two theories can be fruitfully combined has been a major breakthrough in the domain of probabilistic limit theorems and approximations.

\medskip

Since the publication of \cite{StMethOnWienChaos}, the Malliavin-Stein method has generated several hundreds of papers, with ramifications in many (often unexpected) directions, including functional inequalities, random matrix theory, stochastic geometry, non-commutative probability and computer sciences. These developments largely exceed the scope of the present survey, and we invite the interested reader to consult the following references (i)--(vi) for a more detailed presentation: (i) the webpage \cite{WWW} is a constantly updated resource, listing all existing papers written around the Malliavin-Stein method; (ii) the monograph \cite{n-p-book}, written in 2012, contains a self-contained presentation of Malliavin calculus and Stein's method, as applied to functionals of general Gaussian fields, with specific emphasis on random variables belonging to a fixed Wiener chaos; (iii) the text \cite{PRbook} is a collection of surveys, containing an in-depth presentation of variational techniques on the Poisson spaces (including the Malliavin-Stein method), together with their application to asymptotic problems arising in stochastic geometry; (iv) references \cite{MRV21, No20, NPR19l, PR18, PV20, R19, T19} provide a representative overview of applications of Malliavin-Stein techniques to the study of nodal sets associated with Gaussian random fields on two-dimensional manifolds; (v) the papers \cite{NNP21, NPY19} -- and many of the reference therein -- display a pervasive use of Malliavin-Stein techniques to determine rates of convergence in total variation in the Breuer-Major Theorem; (vi) references \cite{CNN20, NN20} deal with the problem of tightness and functional convergence in the Breuer-Major theorem evoked at Point (v).

\medskip

The aim of the present survey is twofolds. On the one hand, we aim at presenting the essence of the Malliavin-Stein's method for functionals of Gaussian fields, by discussing the crucial elements of Malliavin calculus and Stein's method together with their interaction (see Section 2 and Section 3). On the other hand, we aim at introducing the reader to some of the most recent developments on the theory, with specific focus on the general theory of Markov semigroups in a diffusive setting (following the seminal references \cite{ledoux4MT, a-c-p}, as well as \cite{n-p-s, l-n-p-15, l-n-p-16}), and on integration by parts formulae (and associated operators) in the context of functionals of a random point measure \cite{dp, dvz, LPS16, LrSY19, LrPY20+, SY19}. This corresponds to the content of Section 4 and Section 5, respectively.  Finally, Section 6 deals with some recent results (and open problems) concerning $\chi^2$ approximations.
\medskip

From now on, every random object will be defined on a suitable common probability space $(\Omega, \mathscr{F}, P)$, with ${E}$ indicating mathematical expectation with respect to ${P}$. Throughout the paper, the symbol $\mathscr{N}(\mu, \sigma^2)$ will be shorthand for the one-dimensional Gaussian distribution with mean $\mu\in \R$ and variance $\sigma^2>0$. In particular, $X\sim \mathscr{N}(\mu, \sigma^2)$ if and only if
$$
{\P}[X\in A] = \int_A e^{-\frac{(x-\mu)^2}{2\sigma^2}} \frac{dx}{\sqrt{2\pi\sigma^2}},
$$
for every Borel set $A\subset \R$.

\bigskip

\noindent{\bf Acknowledgments}. Giovanni Peccati is supported by the FNR grant {\bf FoRGES} ({\bf R-AGR-3376-10}) at Luxembourg University. Xiaochuan Yang is supported by the EPSRC grant EP/T028653/1.

\section{Elements of Stein's method for normal approximations}\label{sec:stein}
In this section, we briefly introduce the main ingredients of {\bf Stein's method for normal approximations} in dimension one. The approximation will be performed with respect to the \textbf{total variation} and {\bf 1-Wasserstein} distances between the distributions of two random variables; more detailed informations about these distances can be found in \cite[Appendix C]{n-p-book} and the references therein.

\medskip

The crucial intuition behind Stein's method lies in the following heuristic reasoning: {\it it is a well-known fact (see e.g. Lemma \ref{lem:stein}-{(e)} below) that a random variable $X$ has the standard $\mathscr{N}(0,1)$ distribution if and only if 
\begin{equation}\label{e:q}
\E[Xf(X)-f'(X)]=0,
\end{equation}
for every smooth mapping $f: \R\to \R$; heuristically, it follows that, if $X$ is a random variable 
such that the quantity $\E[Xf(X)-f'(X)]$ is close to zero for a large class of test functions $f$, then the distribution of $X$ should be close to Gaussian.}

\medskip

The fact that such a heuristic argument can be made rigorous and applied in a wide array of probabilistic models was the main discovery of Stein's original contribution \cite{Stein72}, where the foundations of Stein's method were first laid. The reader is referred to Stein's monograph \cite{S86}, as well as the books \cite{Chen-book, n-p-book}, for an exhaustive presentation of the theory and its applications (in particular, for extensions to multidimensional approximations).

\medskip

We recall that the total variation distance, between the laws of two real-valued random variables $F$ and $G$, is defined by 
\begin{equation}\label{eq:total-variation}
d_{TV} (F,G) := \sup_{B \in \mathcal{B}(\R)} \Big \vert   \P (F \in B) - \P(G \in B) \Big \vert. 
\end{equation}
One has to note that the topology induced by the distance $d_{TV}$ -- on the set of all probability measures on $\R$ -- is stronger than the topology of convergence in distribution; one sometimes uses the following equivalent representation of $d_{TV}$ (see e.g. \cite[p. 213]{n-p-book}):
\begin{equation}\label{eq:total-equivalent}
d_{TV} (F,G) = \frac{1}{2} \sup \Big\{ \big \vert \E[h(F)] - \E[h(G)] \big \vert  \, : \, h \text{ is Borel measurable and } \Vert h \Vert_\infty \le 1\Big \}.
\end{equation}

\medskip

The 1-Wasserstein distance $d_W$, between the distributions of two real-valued integrable random variables $F$ and $G$, is given by 
\begin{equation}\label{e:w1}
d_{W} (F,G) := \sup_{ h \in {\rm Lip(1)} } \Big \vert   \E [h(F)] - \E[h(G)] \Big \vert,
\end{equation}
where $ {\rm Lip(K)}$, $K>0$ stands for the class of all Lipschitz mappings $h:\R\to \R$ such that $h$ has a Lipschitz constant $\leq K$. As for total variation, the topology induced by $d_{W}$ -- on the set of all probability measures on $\R$ having a finite absolute first moment -- is stronger than the topology of convergence in distribution; it is also interesting to recall the dual representation
\begin{equation}\label{e:w1equiv}
d_{W} (F,G) =  \inf \E\,\big |X-Y \big |,
\end{equation}
where the infimum is taken over all couplings $(X,Y)$ of $F$ and $G$; see e.g. \cite[p. 95]{villani} for a discussion of this fact.
\medskip

The following classical result, whose complete proof can be found e.g. in \cite[p. 64 and p. 67]{n-p-book}, contains all the elements of Stein's method that are needed for our discussion; as for many fundamental findings in the area, such a result can be traced back to \cite{Stein72}.

\begin{lem}\label{lem:stein}
	Let $N \sim \mathscr{N}(0,1)$ be a standard Gaussian random variable. 
	\begin{enumerate}
\item[\rm (a)] Fix $h:\R \to [0,1]$ a Borel-measurable function. Define $f_h: \R \to \R$ as 
\begin{equation}\label{e:fh}
f_h (x):= e^{\frac{x^2}{2}} \int_{-\infty}^{x} \{ h(y) - \E[h(N)]\}   e^{-\frac{y^2}{2}} dy, \quad x \in \R.
\end{equation}
Then, $f_h$ is continuous on $\R$ with $\Vert f_h \Vert_\infty \le \sqrt{\frac{\pi}{2}}$ and $f_h\in {\rm Lip}(2)$. Moreover, there exists a version of $f'_h$ verifying 
\begin{equation}\label{e:steineq} f'_h (x) - x f_h(x) = h(x) - \E[h(N)], \quad \text{ for all } x \in \R.\end{equation}

\item[\rm (b)] Consider $h:\R \to \R\in {\rm Lip}(1)$, and define $f_h: \R \to \R$ as in \eqref{e:fh}.
Then, $f_h$ is of class $C^1$ on $\R$, with $\Vert f'_h \Vert_\infty \le 1$ and $f'_h\in {\rm Lip}(2)$, and $f_h$ solves \eqref{e:steineq}. 

\item[\rm (c)] Let $X$ be an integrable random variable. Then 
$$d_{TV} (X,N) \le \sup_{f} \Big \vert \E\big[f(X)X - f'(X) \big] \Big \vert $$
where the supremum is taken over all pairs $(f,f')$ such that $f$ is a Lipschitz function whose absolute value is bounded by $\sqrt{\frac{\pi}{2}}$, and $f'$ is a version of the derivative of $f$ satisfying $\Vert f' \Vert \le 2$.

\item[\rm (d)] Let $X$ be an integrable random variable. Then,
$$d_{W} (X,N) \le \sup_{f} \Big \vert \E\big[f(X)X - f'(X) \big] \Big \vert $$
where the supremum is taken over all $C^1$ functions $f : \R\to \R$ such that $\Vert f' \Vert \le 2$ and $f'\in {\rm Lip}(2) $.

\item[\rm (e)] Let $X$ be a general random variable. Then $X \sim \mathscr{N}(0,1)$ if and only if $\E[f' (X) - X f(X)]=0$ for every absolutely continuous function $f$ such that $\E \vert f'(N) \vert < +\infty$. 
\end{enumerate}
\end{lem}   

\noindent{\it Sketch of the proof.} Points (a) and (b) can be verified by a direct computation. Point (c) and Point (d) follow by plugging the left-hand side of \eqref{e:steineq} into \eqref{eq:total-equivalent} and \eqref{e:w1}, respectively. Finally, the fact that the relation $\E[f' (X) - X f(X)]=0$ implies that $X \sim \mathscr{N}(0,1)$ is a direct consequence of Point (c), whereas the reverse implication follows by an integration by parts argument. \qed

\section{Normal approximation with Stein's method and Malliavin Calculus}\label{s:malliavin}
The first part of the present section contains some elements of Gaussian analysis and Malliavin calculus. The reader can consult for instance the references \cite{n-p-book, GelbesBuch, Mbook, Nbook} for further details. In Section \ref{ss:m+s} we will shortly explore the connection between Malliavin calculus and the version Stein's method presented in Section \ref{sec:stein}.

\subsection{Isonormal processes, multiple integrals, and the Malliavin operators}\label{ss:isonormal}

Let $ \HH$ be a real separable Hilbert space. For any $q\geq 1$, we write $ \HH^{\otimes q}$ and $ \HH^{\odot q}$ to
indicate, respectively, the $q$th {\bf tensor power} and the $q$th {\bf symmetric tensor power} of $ \HH$; we also set by convention
$ \HH^{\otimes 0} =  \HH^{\odot 0} =\R$. When $\HH = L^2(A,\mathcal{A}, \mu) =:L^2(\mu)$, where $\mu$ is a $\sigma$-finite
and non-atomic measure on the measurable space $(A,\mathcal{A})$, then $ \HH^{\otimes q} \simeq L^2(A^q,\mathcal{A}^q,\mu^q)=:L^2(\mu^q)$, and $ \HH^{\odot q} \simeq L_s^2(A^q,\mathcal{A}^q,\mu^q) := L_s^2(\mu^q)$, 
where $L_s^2(\mu^q)$ stands for the subspace of $L^2(\mu^q)$ composed of those functions that are $\mu^q$-almost everywhere symmetric. We denote by $W=\{W(h) : h\in  \HH\}$
an \textbf{isonormal Gaussian process} over $ \HH$. This means that $W$ is a centered Gaussian family with a covariance structure given by the relation
$\E\left[ W(h)W(g)\right] =\langle h,g\rangle _{ \HH}$. Without loss of generality, we can also assume that $\mathscr{F}=\sigma(W)$, that is, $\mathscr{F}$ is generated by $W$, and use the shorthand notation $L^2(\Omega) := L^2(\Omega, \mathscr{F}, \P)$.

\medskip

For every $q\geq 1$, the symbol $C_{q}$ stands for the $q$th \textbf{Wiener chaos} of $W$, defined as the closed linear subspace of $L^2(\Omega)$
generated by the family $\{H_{q}(W(h)) : h\in  \HH,\left\| h\right\| _{ \HH}=1\}$, where $H_{q}$ is the $q$th {\bf Hermite polynomial}, defined as follows:
\begin{equation}\label{hq}
H_q(x) = (-1)^q e^{\frac{x^2}{2}}
\frac{d^q}{dx^q} \big( e^{-\frac{x^2}{2}} \big).
\end{equation}
We write by convention $C_{0} = \mathbb{R}$. For
any $q\geq 1$, the mapping $I_{q}(h^{\otimes q})=H_{q}(W(h))$ can be extended to a
linear isometry between the symmetric tensor product $ \HH^{\odot q}$
(equipped with the modified norm $\sqrt{q!}\left\| \cdot \right\| _{ \HH^{\otimes q}}$)
and the $q$th Wiener chaos $C_{q}$. For $q=0$, we write by convention $I_{0}(c)=c$, $c\in\mathbb{R}$. 

\medskip 

It is well-known that $L^2(\Omega)$ can be decomposed into the infinite orthogonal sum of the spaces $C_{q}$: this means that any square-integrable random variable
$F\in L^2(\Omega)$ admits the following \textbf{Wiener-It\^{o} chaotic expansion}
\begin{equation}
F=\sum_{q=0}^{\infty }I_{q}(f_{q}),  \label{E}
\end{equation}
where the series converges in $L^2(\Omega)$, $f_{0}=E[F]$, and the kernels $f_{q}\in  \HH^{\odot q}$, $q\geq 1$, are
uniquely determined by $F$. For every $q\geq 0$, we denote by $J_{q}$ the
orthogonal projection operator on the $q$th Wiener chaos. In particular, if
$F\in L^2(\Omega)$ has the form (\ref{E}), then
$J_{q}F=I_{q}(f_{q})$ for every $q\geq 0$.

\medskip

Let $\{e_{k},\,k\geq 1\}$ be a complete orthonormal system in $\HH$. Given $f\in  \HH^{\odot p}$ and $g\in \HH^{\odot q}$, for every
$r=0,\ldots ,p\wedge q$, the \textbf{contraction} of $f$ and $g$ of order $r$
is the element of $ \HH^{\otimes (p+q-2r)}$ defined by
\begin{equation}
f\otimes _{r}g=\sum_{i_{1},\ldots ,i_{r}=1}^{\infty }\langle
f,e_{i_{1}}\otimes \ldots \otimes e_{i_{r}}\rangle _{ \HH^{\otimes
		r}}\otimes \langle g,e_{i_{1}}\otimes \ldots \otimes e_{i_{r}}
\rangle_{ \HH^{\otimes r}}.  \label{v2}
\end{equation}
Notice that the definition of $f\otimes_r g$ does not depend
on the particular choice of $\{e_k,\,k\geq 1\}$, and that
$f\otimes _{r}g$ is not necessarily symmetric; we denote its
symmetrization by $f\widetilde{\otimes }_{r}g\in  \HH^{\odot (p+q-2r)}$.
Moreover, $f\otimes _{0}g=f\otimes g$ equals the tensor product of $f$ and
$g$ while, for $p=q$, $f\otimes _{q}g=\langle f,g\rangle _{ \HH^{\otimes q}}$. 
When $\HH = L^2(A,\mathcal{A},\mu)$ and $r=1,...,p\wedge q$, the contraction $f\otimes _{r}g$ is the element of $L^2(\mu^{p+q-2r})$ given by
\begin{eqnarray}\label{e:contraction}
&& f\otimes _{r}g (x_1,...,x_{p+q-2r})\\
&& = \int_{A^r} f(x_1,...,x_{p-r},a_1,...,a_r)\times \notag\\
&& \quad\quad\quad\quad \times g(x_{p-r+1},...,x_{p+q-2r},a_1,...,a_r)d\mu(a_1)...d\mu(a_r). \notag
\end{eqnarray}

\medskip

It is a standard fact of Gaussian analysis that the following \textbf{multiplication formula} holds: if $f\in  \HH^{\odot p}$ and $g\in  \HH^{\odot q}$, then
\begin{eqnarray}\label{multiplication}
I_p(f) I_q(g) = \sum_{r=0}^{p \wedge q} r! {p \choose r}{ q \choose r} I_{p+q-2r} (f\widetilde{\otimes}_{r}g).
\end{eqnarray}
\smallskip


We now introduce some basic elements of the Malliavin calculus with respect
to the isonormal Gaussian process $W$. 

\medskip

Let $\mathcal{S}$
be the set of all
cylindrical random variables of
the form
\begin{equation}
F=g\left( W(\phi _{1}),\ldots ,W(\phi _{n})\right) ,  \label{v3}
\end{equation}
where $n\geq 1$, $g:\mathbb{R}^{n}\rightarrow \mathbb{R}$ is an infinitely
differentiable function such that its partial derivatives have polynomial growth, and $\phi _{i}\in  \HH$,
$i=1,\ldots,n$.
The \textbf{Malliavin derivative}  of $F$ with respect to $W$ is the element of $L^2(\Omega , \HH)$ defined as
\begin{equation*}
DF\;=\;\sum_{i=1}^{n}\frac{\partial g}{\partial x_{i}}\left( W(\phi_{1}),\ldots ,W(\phi _{n})\right) \phi _{i}.
\end{equation*}
In particular, $DW(h)=h$ for every $h\in  \HH$. By iteration, one can define the $m${\bf th derivative} $D^{m}F$, which is an element of $L^2(\Omega , \HH^{\odot m})$,
for every $m\geq 2$.
For $m\geq 1$ and $p\geq 1$, ${\mathbb{D}}^{m,p}$ denotes the closure of
$\mathcal{S}$ with respect to the norm $\Vert \cdot \Vert _{m,p}$, defined by
the relation
\begin{equation*}
\Vert F\Vert _{m,p}^{p}\;=\; \E\left[ |F|^{p}\right] +\sum_{i=1}^{m} \E\left[
\Vert D^{i}F\Vert _{ \HH^{\otimes i}}^{p}\right].
\end{equation*}
We often use the (canonical) notation $\mathbb{D}^{\infty} := \bigcap_{m\geq 1}
\bigcap_{p\geq 1}\mathbb{D}^{m,p}$. For example, it is a well-known fact that any random variable $F$ that is a finite linear combination of multiple Wiener-It\^o integrals is an element of $\mathbb{D}^\infty$. The Malliavin derivative $D$ obeys the following \textbf{chain rule}. If
$\varphi :\mathbb{R}^{n}\rightarrow \mathbb{R}$ is continuously
differentiable with bounded partial derivatives and if $F=(F_{1},\ldots
,F_{n})$ is a vector of elements of ${\mathbb{D}}^{1,2}$, then $\varphi
(F)\in {\mathbb{D}}^{1,2}$ and
\begin{equation}\label{e:chainrule}
D\,\varphi (F)=\sum_{i=1}^{n}\frac{\partial \varphi }{\partial x_{i}}(F)DF_{i}.
\end{equation}

\medskip

Note also that a random variable $F$ as in (\ref{E}) is in ${\mathbb{D}}^{1,2}$ if and only if
$\sum_{q=1}^{\infty }q\|J_qF\|^2_{L^2(\Omega)}<\infty$
and in this case one has the following explicit relation: $$\E\left[ \Vert DF\Vert _{ \HH}^{2}\right]
=\sum_{q=1}^{\infty }q\|J_qF\|^2_{L^2(\Omega)}.$$ If $ \HH=
L^{2}(A,\mathcal{A},\mu )$ (with $\mu $ non-atomic), then the
derivative of a random variable $F$ as in (\ref{E}) can be identified with
the element of $L^2(A \times \Omega )$ given by
\begin{equation}
D_{t}F=\sum_{q=1}^{\infty }qI_{q-1}\left( f_{q}(\cdot ,t)\right) ,\quad t \in A.  \label{dtf}
\end{equation}


The operator $\LL$, defined as $\LL=\sum_{q=0}^{\infty }-qJ_{q}$, is the \textbf{infinitesimal generator of the Ornstein-Uhlenbeck semigroup}. The domain of $\LL$ is
\begin{equation*}
\mathrm{Dom}\LL=\{F\in L^2(\Omega ):\sum_{q=1}^{\infty }q^{2}\left\|
J_{q}F\right\| _{L^2(\Omega )}^{2}<\infty \}=\mathbb{D}^{2,2}\text{.}
\end{equation*}


For any $F \in L^2(\Omega )$, we define $\LL^{-1}F =\sum_{q=1}^{\infty }-\frac{1}{q} J_{q}(F)$. The operator $\LL^{-1}$ is called the
\textit{pseudo-inverse} of $\LL$. Indeed, for any $F \in L^2(\Omega )$, we have that $\LL^{-1} F \in  \mathrm{Dom}\LL
= \mathbb{D}^{2,2}$,
and
\begin{equation}\label{Lmoins1}
\LL \LL^{-1} F = F - \E(F).
\end{equation}

The following infinite dimensional Malliavin integration by parts formula plays a crucial role in the analysis (see for instance \cite[Section 2.9]{n-p-book} for a proof).

\begin{lem}\label{L : Tech1}
	Suppose that $F\in\mathbb{D}^{1,2}$ and $G\in L^2(\Omega)$. Then, ${\rm \LL}^{-1}G \in \mathbb{D}^{2,2}$ and
	\begin{equation}
	\E[FG] = \E[F]\E[G]+\E[\langle DF,-D\LL^{-1}G\rangle_{\HH}].
	\end{equation}
\end{lem}
Inspired by the Malliavin integration by parts formula appearing in Lemma \ref{L : Tech1}, we now introduce a class of \textbf{iterated Gamma operators}. We will need such operators in Section 
\ref{sec:MS2W}.
 \begin{Def}[See Chapter 8 in \cite{n-p-book}]\label{Def : Gamma} {\rm Let $F\in
	\mathbb{D}^{\infty}$; the sequence of random variables
	$\{\Gamma_i(F)\}_{i\geq 0}\subset \mathbb{D}^\infty$ is recursively
	defined as follows. Set $\Gamma_0(F) = F$ and, for every $i\geq
	1$,
	\[\Gamma_{i}(F) = \langle DF,-DL^{-1}\Gamma_{i-1}(F)\rangle_{\HH}.
	\]
}
\end{Def}

\begin{Def}[Cumulants]\label{D : cum}{\rm Let $F$ be a real-valued random variable such that $\E|F|^m<\infty$ for some integer
		$m\geq 1$, and write $\phi_F(t) = \E[e^{itF}]$, $t\in\R$, for the characteristic function of $F$.
		Then, for $r=1,...,m$, the $r$th \textbf{cumulant} of $F$, denoted by $\kappa_r(F)$, is given by
		\begin{equation}
		\kappa_r (F) = (-i)^r \frac{d^r}{d t^r} \log \phi_F (t)|_{t=0}.
		\end{equation}}
\end{Def}

\begin{Rem}{\rm When $\E(F)=0$, then the first four cumulants of $F$ are the following: $\kappa_1(F) = \E[F]=0$, $\kappa_2(F) = \E[F^2]= \Var(F)$,
		$\kappa_3(F) = \E[F^3]$, and \[\kappa_4(F) = \E[F^4] - 3\E[F^2]^2.
		\]
	}
\end{Rem}
The following statement explicitly connects the expectation of the random variables $\Gamma_r(F)$ to the cumulants of $F$. 
\begin{Prop}[See Chapter 8 in \cite{n-p-book}]
Let $F\in \mathbb{D}^{\infty}$. Then $\kappa_r(F)= (r-1)!\E[\Gamma_{r-1}(F)]$  for every $r \ge 1$.
\end{Prop}

As announced, in the next subsection we show how to use the above Malliavin machinery in order to study the Stein's bounds presented in Section \ref{sec:stein}. 

\subsection{Connection with Stein's method}\label{ss:m+s}
Let $F \in\mathbb{D}^{1,2}$ with $\E[F]=0$ and $\E[F^2]=1$. Take a $C^1$ function such that $\Vert f \Vert \le \sqrt{\frac{\pi}{2}}$ and $\Vert f' \Vert \le 2$. Using the Malliavin integration by parts formula stated in Lemma \ref{L : Tech1} together with the chain rule \eqref{e:chainrule}, we can write 
\begin{equation}
\begin{split}
\Big \vert \E[f'(F) - F f(F)] \Big \vert &= \Big \vert \E[f'(F) \left(    1- \langle DF,-D\LL^{-1}F\rangle_{\HH} \right)]\Big \vert \label{e:ibpg} \\
&\le 2 \, \E \Big\vert  1- \langle DF,-D\LL^{-1}F\rangle_{\HH} \Big \vert. \notag
\end{split}
\end{equation}
If we furthermore assume that $F \in\mathbb{D}^{1,4}$, then the random variable $ 1- \langle DF,-D\LL^{-1}F\rangle_{\HH}$ is square-integrable, using the Cauchy-Schwarz inequality we infer that 
$$\Big \vert \E[f'(F) - F f(F)] \Big \vert  \le 2 \sqrt{\Var \left( \langle DF,-D\LL^{-1}F\rangle_{\HH} \right)}.$$
Note that in above we used the fact that $\E[\langle DF,-D\LL^{-1}F\rangle_{\HH}]= \E[F^2]=1$.  The above arguments combined with Lemma \ref{lem:stein} yield immediately \footnote{This is not completely accurate: attention has indeed to be paid to the fact that the function $f_h$ in \eqref{e:steineq} is only almost everywhere differentiable, and $F$ does not necessarily have a density -- see \cite[Theorem 5.2]{Nlecturenotes} for a detailed proof based on Lusin Theorem.} the next crucial statement, originally proved in \cite{StMethOnWienChaos}.

\begin{Thm}\label{t:gb}
Let $F \in\mathbb{D}^{1,2}$ be a generic random element with $\E[F]=0$ and $\E[F^2]=1$. Let $N \sim \mathscr{N}(0,1)$. Assume further that $F$ has a density with respect to the Lebesgue measure. Then, 
$$d_{TV}(F,N) \le 2\, \E \Big\vert  1- \langle DF,-D\LL^{-1}F\rangle_{\HH} \Big \vert.$$
Moreover, assume that $F \in\mathbb{D}^{1,4}$, then $$ d_{TV}(F,N) \le 2   \sqrt{\Var \left( \langle DF,-D\LL^{-1}F\rangle_{\HH} \right)}.$$
In particular case, if $F=I_q(f)$ belongs to the Wiener chaos of order $q \ge 2$, then
\begin{equation}\label{eq:4mt-bound}
d_{TV}(F,N) \le 2   \sqrt{\frac{q-1}{3q}\Big( \E[F^4] - 3\Big)}.
\end{equation}
\end{Thm} 

Note that, by virtue of Lemma \ref{lem:stein}, similar bounds can be immediately obtained for the Wasserstein distance $d_W$ (and many more -- see \cite[Chapter 5]{n-p-book}). In particular, the previous statement allows one to recover the following central limit theorem for chaotic random variables, first proved in \cite{FmtOriginalReference}.

\begin{Cor}[Fourth Moment Theorem]\label{c:fmt}
	Let $\{F_n \}_{ n \ge 1}=\{ I_q(f_n) \}_{ n\ge1}$ be a sequence of random elements in a fixed Wiener chaos of order $q \ge2$ such that $\E[F^2_n] = q! \Vert f_n\Vert^2 =1$. Assume that $N \sim \mathscr{N}(0,1)$. Then, as $n$ tends to infinity, the following assertions are equivalent.
	\begin{description}
		\item[(I)] $F_n \longrightarrow N$ in distribution.
		\item[(II)] $\E[F^4_n] \longrightarrow 3 \, (= \E[N^4])$. 
		\end{description}
\end{Cor}

As demonstrated by the webpage \cite{WWW}, the `fourth moment theorem' stated in Corollary \ref{c:fmt} has been the starting point of a very active line of research, composed of several hundreds papers connected with disparate applications. In the next section, we will implicitly provide a general version of Theorem \ref{t:gb} (with the 1-Wasserstein distance replacing the total variation distance), whose proof relies only on the spectral properties of the Ornstein-Uhlenbeck generator $\LL$ and on the so-called $\Gamma$ calculus (see e.g. \cite{b-g-l}).

\section{The Markov triple approach}\label{s:mt}
In this section, we introduce a general framework for studying and generalizing the fourth moment phenomenon appearing in the statement of Corollary \ref{c:fmt}. The forthcoming approach was first introduced in \cite{ledoux4MT} by M. Ledoux, and then further developed and generalizes in \cite{a-c-p, a-m-m-p}. 
\subsection{Diffusive fourth moment structures}
We start with definition of our general setup.
\begin{Def}\label{d:dfms}{\rm 
	A {\bf diffusive fourth moment structure} is a triple $(E,\mu,\LL)$ such that:
	\begin{itemize}
		\item[(a)] $(E,\mu)$ is a probability space;
		\item[(b)] $\LL$ is a symmetric unbounded operator defined on some dense subset of $L^2(E,\mu)$, that we denote by $\mathcal{D}(\LL)$ (the set $\mathcal{D}(\LL)$ is called the {\bf domain} of $\LL$);
		\item[(c)] the associated {\bf carr\'e-du-champ operator} $\Gamma$ is a symmetric bilinear operator, and is defined by
		\begin{equation}\label{Gamma}
		2 \Gamma\left[X,Y\right] := \LL \left[XY\right] - X \LL \left[Y\right] - Y \LL \left[X\right];
		\end{equation}
		\item[(d)] the operator $\LL$ is {\bf diffusive}, meaning that, for any $\mathcal{C}^2_b$ function $\phi \colon \R \to \R$, any $X \in \mathcal{D}(\LL)$, it holds that $\phi(X)\in\mathcal{D}(\LL)$ and
		\begin{equation}\label{diff}
		\LL \left[\phi(X)\right] = \phi'(X) \LL [X] + \phi''(X) \Gamma[X,X];
		\end{equation}
		Note that $\LL[1]=0$ (by taking $\phi=1 \in \mathcal{C}^2_b$).  The latter property is equivalent to say that operator $\Gamma$ satisfies in the chain rule: 
		$$\Gamma \left[ \phi(X),X \right] = \phi'(X) \Gamma[X,X];$$
		
		\item[(e)]  the operator $-\LL$ diagonalizes the space $L^2(E,\mu)$ with $\textbf{sp}(-\LL)=\N$, meaning that 
		\begin{equation*}
		L^2(E,\mu)=\bigoplus_{i=0}^\infty \K(\LL+i \Id);
		\end{equation*}
		\item[(f)] for any pair of eigenfunctions $(X,Y)$ of the operator $- \LL$ associated with the eigenvalues $(p_1,p_2)$,\begin{equation}\label{fundamental-assumtionbis}
		X Y\in \bigoplus_{i \le p_1+p_2  } \K \left( \LL +  i \Id \right).
		\end{equation}
	\end{itemize}
	}
\end{Def}
In this context, we usually write $\Gamma[X]$ instead of $\Gamma[X,X]$ and $\E$ denotes the integration against probability measure $\mu$.

\begin{Rem}{\rm 
	\begin{enumerate}
\item[(1)]	Property (d) together with symmetric property of the operator $\LL$ determine a functional calculus through the following fundamental \textbf{integration by parts formula}: for any $X,Y$ in $\mathcal{D}(\LL)$ and 
	$\phi\in\mathcal{C}^2_b$,
	\begin{equation}\label{by-parts}
	\E \left[\phi'(X)\Gamma\left[X,Y\right]\right] = - \E\left[\phi(X) \LL \left[Y\right]\right] = -\E\left[Y \LL \left[\phi(X)\right]\right].
	\end{equation}
	\item[(2)] The results in this section can be stated under the weaker assumption that $\textbf{sp}(-\LL) =\{ 0 = \lambda_0 < \lambda_1, \cdots,\lambda_k < \cdots\} \subset \R_+$ is discrete. However, to keep a transparent presentation, we restrict ourself to the 
	assumption $\textbf{sp}(-\LL)=\N$. The reader is referred to \cite{a-c-p} for further details.\\
	\item[(3)] We point out that, by a recursive argument, assumption $(\ref{fundamental-assumtionbis})$ yields that for any $X\in\K(\LL+p\Id)$ and any polynomial $P$ of degree $m$, we have 
	\begin{equation}\label{fundamental-assumtion}
	P(X) \in \bigoplus_{i \le mp  } \K \left( \LL +  i \Id \right).
	\end{equation}

\item[(4)] The eigenspaces of a diffusive fourth moment structure are \textbf{hypercontractive} (see \cite{bakrey} for details and sufficient conditions), that is, there exists a constant $C(M,k)$ such that for any $X\in\bigoplus_{i \le M} \K \left( \LL + i \Id \right)$:
		\begin{equation}\label{Hyper2}
		\E(X^{2k})\leq C(M,k) ~\E(X^2)^k.
		\end{equation}

\item[(5)] Property (f) in the previous definition roughly implies that eigenfunctions of $\LL$ in a diffusive fourth moment structure behave like orthogonal polynomial with respect to multiplication. 	
	\end{enumerate}}	
		
\end{Rem}

For further details on our setup, we refer the reader to \cite{b-g-l} as well as  \cite{a-c-p,a-m-m-p}. The next example describes some diffiusive fourth moment structures. The reader can consult \cite[Section 2.2]{a-m-m-p} for two classical methods for building {further} diffusive fourth moment structures {starting from known ones}.

\begin{Ex}{\rm 
	\begin{itemize}
	\item[(a)] \textbf{ Finite-Dimensional Gaussian Structures:} Let $d\geq 1$ and denote by $\gamma_d$ the 
	$d$-dimensional standard Gaussian measure on $\R^d$. It is well known (see for
	example~\cite{b-g-l}), that $\gamma_d$ is the invariant measure of the Ornstein-Uhlenbeck generator, defined
	for any test function $\phi$ by 
	\begin{equation}\label{O-U} \LL \phi(x)
	=\Delta\phi-\sum_{i=1}^{d}x_i \partial_i \phi(x).
	\end{equation} Its spectrum is given by $- \N_0$ and the eigenspaces are of the
	form 
	\begin{equation*} \textbf{Ker}(\LL+k\Id)=
	\left\{\sum_{i_1+i_2+\cdots+i_{d}=k}\alpha(i_1,\cdots,i_{d})\prod_{j=1}^{d}
	H_{i_j}(x_j)\right\},
	\end{equation*} where $H_n$ denotes the Hermite polynomial of order 
	$n$. Since, eigenfunctions of $\LL$ are multivariate polynomials so it is straightforward to see that assumption (f) is also verified.\\ 
	
	{ \item[(b)]\textbf{Wiener space and isonormal processes:} Letting $d\to \infty$ in the setup of the previous item (a) one recovers the infinite dimensional generator of the Ornstein-Uhlenbeck semigroup for isonormal processes, as defined in Section \ref{ss:isonormal}. It is easily verified in particular, by using \eqref{multiplication} that $(\Omega, \mathscr{F}, \LL)$ is also a diffusive fourth moment structure.
	
	}
	\item[(c)]\textbf{Laguerre Structure:}	Let $\nu \geq -1$, and $ \pi_{1,\nu}(dx) =
	x^{\nu-1}\frac{\me^{-x}}{\Gamma(\nu)} \textbf{1}_{(0,\infty)} \ud x$ be the Gamma
	distribution with parameter $\nu$ on $\R_+$. The associated Laguerre generator is
	defined for any test function $\phi$ (in dimension one) by: 
	\begin{equation}\label{Lag1} \LL_{1,\nu} (\phi)= x\phi''(x)+(\nu+1-x)\phi'(x).
	\end{equation}
	By a classical tensorization procedure, we obtain the Laguerre generator in
	dimension $d$ associated with the measure $ \pi_{d,\nu}(\ud x) = \pi_{1,\nu}(\ud x_1) 
	\pi_{1,\nu}(\ud x_2) \cdots \pi_{1,\nu}(\ud x_d)$, where $x=(x_1,x_2, \cdots,x_d)$.
	\begin{equation}\label{Lag2} \LL_{d,\nu}(\phi) = \sum_{i=1}^{d}
	\Big{(}x_{i} \partial_{i,i}\phi+(\nu+1-x_i)\partial_i \phi\Big{)}
	\end{equation}
	
	It is also classical that (see for example~\cite{b-g-l}) the spectrum of
	$\LL_{d,\nu}$ is given by~$-\N_0$ and moreover that
	\begin{equation} \textbf{Ker}(\LL_{d,p} + k\Id) =
	\left\{\sum_{i_1+i_2+\cdots+i_{d}=k} \alpha(i_1,\cdots,i_{d})\prod_{j=1}^{d}
	L^{(\nu)}_{i_j}(x_j)\right\},
	\end{equation}
	where $L^{(\nu)}_n$ stands for the Laguerre polynomial of order $n$ with parameter $\nu$ which is defined by
	$$ L_n^{(\nu)}(x)= {x^{-\nu} e^x \over n!}{d^n \over dx^n} \left(e^{-x} x^{n+\nu}\right).$$
	\end{itemize}}
\end{Ex}

{In the next subsection, we demonstrate how a diffusive fourth moment structure can be combined with the tools of $\Gamma$ calculus, in order to deduce substantial generalizations of Theorem \ref{t:gb}.}

\subsection{Connection with $\Gamma$ Calculus}
Throughout this section, we assume that $(E,\mu,\LL)$ is a diffiusive fourth moment structure. Our principal aim is to prove an analogous fourth moment criterion to that of \eqref{eq:4mt-bound} for eigenfunctions of the operator $\LL$. To do this, we assume that $X \in \K (\LL + q \Id)$ for some $q\ge 1$ with $\E[X^2]=1$. {The arguments implemented in the proof will clearly demonstrate that requirements (d) and (f) in Definition \ref{d:dfms} are the most crucial elements in order to establish our estimates.}

\begin{prop}\label{prop:variance-estimate}
Let $q \ge 1$. Assume that $X \in \K (\LL + q \Id)$ with $\E[X^2]=1$. Then,
\begin{equation*}
\Var \left(\Gamma[X]\right) \le \frac{q^2}{3} \left\{ \E[X^4] - 3\right\}.
\end{equation*}
\end{prop}

\begin{proof}
First note that by using integration by parts formula \eqref{by-parts}, we have $\E[\Gamma[X]] = -\E[X \LL X] = q \E[X^2]=q$. Secondly, by using the definition of the 
carr\'e-du-champ operator $\Gamma$ and the fact that $\LL X = -q X$, one easily verifies that $$\Gamma[X] - q = \frac{1}{2} \left(  \LL + 2q \Id \right) (X^2-1).$$ Next, taking into account properties (f) and (g) we can conclude that $$X^2-1 \in \bigoplus_{ 1 \le i \le 2q} \K \left( \LL + i \Id \right).$$ For the rest of the proof, we use the notation $J_i$ to denote the projection of a square-integrable element $X$ over the eigenspace $\K \left( \LL + i \Id\right)$. Now,
\begin{equation*}
\begin{split}
\Var \left( \Gamma[X]\right) &= \E \left[\left( \Gamma[X]  - q \right)^2\right] = \frac{1}{4} \E \left[  \left(  \LL + 2q \Id \right) (X^2-1) \times \left(  \LL + 2q \Id \right) (X^2-1) \right]\\
&= \frac{1}{4} \E \left[   \LL (X^2-1) \left(  \LL + 2q \Id \right) (X^2-1) \right]+ \frac{q}{2} \E \left[ (X^2-1) \left(  \LL + 2q \Id \right) (X^2-1) \right]\\
&= \frac{1}{4} \sum_{1 \le i \le 2q} (-i)(2q-i) \E \left[ \left(J_i (X^2-1)   \right)^2\right]+ \frac{q}{2} \E \left[ (X^2-1) \left(  \LL + 2q \Id \right) (X^2-1) \right]\\
&\le  \frac{q}{2} \E \left[ (X^2-1) \left(  \LL + 2q \Id \right) (X^2-1) \right] \\
& =  q \E \left[ (X^2-1) (\Gamma[X] - q)\right] = q \E \left[ (X^2-1) \Gamma[X] \right]\\
& = q \E \left[ \Gamma [\frac{X^3}{3} - X,X] \right]= - q \E \left[ \left( \frac{X^3}{3} - X \right) \LL X \right]\\
&= q^2 \E\left[X \left(\frac{X^3}{3} - X \right)\right] = q^2 \E\left[ \frac{X^4}{3} - X^2\right]\\
&= \frac{q^2}{3} \left\{ \E[X^4] -3\right\},
	\end{split}.
\end{equation*}    
thus yielding the desired conclusion.
\end{proof}

In order to avoid some technicalities, we now present a quantitative bound in the 1-Wasserstein distance $d_W$ (and not in the more challenging total variation distance $d_{TV}$) for eigenfunctions of the operator $\LL$. This requires to adapt the Stein's method machinery presented in Section \ref{sec:stein} to our setting, as a direct application of the integration by part formula \eqref{by-parts}. The arguments below are borrowed in particular from \cite[Proposition 1]{ledoux4MT}.  

\begin{prop}\label{prop:general-setting}
Let $(E,\mu,\LL)$ be a diffiusive fourth moment structure. Assume that $X \in \K (\LL + q \Id)$ for some $q \ge 1$ with $\E[X^2]=1$. Let $N \sim \mathscr{N}(0,1)$. Then, 

\begin{equation*}
d_{W}(X, N)	\le \frac{2}{q}  \Var \left(\Gamma[X]\right)^\frac{1}{2}.
\end{equation*}
\end{prop}
\begin{proof}
For every function $f$ of class $C^1$ on $\R$, with $\Vert f' \Vert_\infty \le 1$ and $f'\in {\rm Lip}(2)$ according to Part (b) in Lemma \ref{lem:stein}, it is enough to show that 
$$\Big \vert   \E \left[  f'(X) - X f(X) \right] \Big \vert 	\le \frac{2}{q}  \Var \left(\Gamma[X]\right)^\frac{1}{2}.$$ Since $\LL X = -q X$, and diffusivity of the operator $\Gamma$ together with integration by parts formula \ref{by-parts}, one can write that 
\begin{equation*}
\begin{split}
  \E \left[  f'(X) - X f(X) \right] &=  \E \left[  f'(X) + \frac{1}{q} \LL(X) f(X) \right]=  \E \left[  f'(X) - \frac{1}{q} \Gamma[f(X),X] \right]\\
  &=  \E \left[  f'(X) - \frac{1}{q} f'(X) \Gamma[X] \right]\\
  & = \frac{1}{q} \E \left[  f'(X)  \left(    q - \Gamma[X] \right)  \right].
  \end{split}
  \end{equation*}
  Now, the claim follows at once by using the Cauchy-Schwarz inequality and noting that $\E[\Gamma[X]] = q \, \E[X^2] = q$.
\end{proof}
We end this section with the following general version of the fourth moment theorem for eigenfunctions of the operator $\LL$, obtained by combining Propositions \ref{prop:variance-estimate} and \ref{prop:general-setting}.

\begin{Thm}\label{Thm:4mt-general}
	Let $(E,\mu,\LL)$ be a diffiusive fourth moment structure. Assume that $X \in \K (\LL + q \Id)$ for some $q \ge 1$ with $\E[X^2]=1$. Let $N \sim \mathscr{N}(0,1)$. Then, 
	$$ d_{W}(X, N) \le \frac{2}{\sqrt{3}} \, \sqrt{ \E[X^4] - 3 }.$$
It follows that, if $\{ X_n \}_{n \ge 1}$ is a sequence of eigenfunctions in a fixed eigenspace $ \K (\LL + q \Id)$ where $q \ge 1$ and $\E[X^2_n]=1$ for all $n \ge 1$, then the following implication holds: $\E[X^4_n] \to 3$ if and only if $X_n$ converges in distribution towards the standard Gaussian random variable $N$.  
	\end{Thm}

\begin{Rem}{\rm
The fact that the condition $\E[X^4_n] \to 3$ is necessary for convergence to Gaussian is a direct consequence of the hypercontractive estimate \eqref{Hyper2}.
}
\end{Rem}

\subsection{Transport distances, Stein discrepancy and $\Gamma$ Calculus}\label{sec:TD-SD-GC}
The general setting of the Markov triple together with $\Gamma$ calculus provide a suitable framework to study \textbf{functional inequalities} such as the classical \textbf{logarithmic Sobolev inequality} or the celebrated \textbf{Talagrand quadratic transportation cost inequality}. For simplicity, here we restrict ourself to the setting of Wiener structure and the Gaussian measure to be our reference measure. The reader may consult references \cite{l-n-p-15,l-n-p-16} for a presentation of the general setting, and \cite{Entropy, n-p-s} for some previous references connecting fourth moment theorems and entropic estimates.\\ 
  
Let $d \ge 1$, and $ d\gamma(x) = (2\pi)^{- \frac{d}{2}} e^{- \frac{\vert x \vert }{2}} dx$ be the standard Gaussian measure on $\R^d$. Assume that $d\nu = h d\gamma$ is a probability measure on $\R^d$ with a (smooth) density function $h : \R^d \to \R_+$ with respect to the Gaussian measure $\gamma$. Inspired from Gaussian integration by parts formula we introduce first the crucial notion of a \textbf{Stein kernel} $\tau_\nu$ associated with the probability measure $\nu$ and, then, the concept of \textbf{Stein discrepancy}.

\begin{Def}{\rm
	(a) A measurable matrix-valued map $\tau_\nu$ on $\R^d$ is called a {\bf Stein kernel} for the centered probability measure $\nu$ if for every smooth test function $\varphi : \R^d \to \R$,
	$$ \int_{\R^d} x \cdot \nabla \varphi d\nu = \int_{\R^d} \langle \tau_\nu,\Hess (\varphi) \rangle_{\HS} d\nu,$$
	where $\Hess(\varphi)$ stands for the Hessian of $\varphi$, and $\langle \, {,} \, \rangle_{\HS}$, and $\Vert \, {,} \,  \Vert_{\HS}$ denote the usual Hilbert-Schmidt scalar product and norm, respectively. \\
	(b) The {\bf Stein discrepancy} of $\nu$ with respect to $\gamma$ is defined as $$\SSS(\nu,\gamma) = \inf \Big( \int_{\R^d} \Vert \tau_\nu - \Id \Vert_{\HS}^2 d\nu     \Big)^\frac{1}{2}$$
	where the infimum is taken over all Stein kernels of $\nu$, and takes the value $+\infty$ if a Stein kernel for $\nu$ does not exist. 
} 
\end{Def}
We recall that the Stein kernel $\tau_\nu$ is uniquely defined in dimension $d=1$, and that unicity may fail in higher dimensions $d\geq 2$, see \cite[Appendix A]{n-p-s}. Also, $\tau_\gamma = \Id_{d \times d}$ the identity matrix. We further refer to \cite{Fathi,c-f-p} for existence of the Stein kernel in general settings.  The interest of the Stein's discrepancy comes e.g. from the fact that -- as a simple application of Stein's method --
\begin{equation*}
d_{TV} (\nu,\gamma) \le 2 \int_\R \vert \tau_\nu - 1 \vert d\nu \le 2 \Big(  \int_\R \vert \tau_\nu - 1 \vert^2 d\nu  \Big)^\frac{1}{2},
\end{equation*} 
yielding that $d_{TV}(\nu,\gamma) \le 2 \SSS(\nu,\gamma)$; see \cite{l-n-p-15} for further details.

\medskip

Next, we need the notion of Wasserstein distance. Let $p \ge 1$. Given two probability measure $\nu$ and $\mu$ on the Borel sets of $\R^d$, whose marginals have finite moments of order $p$, we define the \textbf{$p$-Wasserstein distance} between $\nu$ and $\mu$ as 
\begin{equation*}
\W_p (\nu,\mu) = \inf_{\pi} \Big(  \int_{\R^d \times \R^d} \vert x-y \vert^p d\pi(x,y) \Big)^\frac{1}{p}
\end{equation*}
where the infimum is taken over all probability measures $\pi$ of $\R^d \times \R^d$ with marginals $\nu$ and $\mu$; note that $ {\rm W}_1= d_W$, as defined in Section 2. 

\medskip

For a measure $\nu=h \gamma$ with a smooth density function $h$ on $\R^d$, we recall that $$\HHH(\nu,\gamma):= \int_{\R^d} h \log h d\gamma = \Ent_\gamma(h)$$ is the \textbf{relative entropy} of the measure $\nu$ with respect to $\gamma$, and $$\III(\nu,\gamma):= \int_{\R^d} \frac{\vert \nabla h\vert^2}{h} d\gamma$$ is the \textbf{Fisher information} of $\nu$ with respect to $\gamma$. After having established these notions, we can state two popular probabilistic/Entropic functional inequalities :\\
\begin{enumerate}
\item [\rm (i)][\textbf{Logarithmic Sobolev inequality}]: $ \qquad  \HHH(\nu,\gamma) \le \frac{1}{2} \III(\nu,\gamma)$.
\item[\rm (ii)] [\textbf{Talagrand quadratic transportation cost inequality}]: $ \qquad \W^2_2(\nu,\gamma) \le 2 \HHH (\nu,\gamma)$.
\end{enumerate}

The next theorem is borrowed from \cite{l-n-p-15}, and represents a significant improvement of the previous logarithmic Sobolev and Talagrand inequalities based on the use of Stein discrepancies: the techniques used in the proof are based on an interpolation argument along the Ornstein-Uhlenbeck semigroup. The reader is also referred to the recent works \cite{Fathi,c-f-p,saumard} for related estimates of the Stein discrepancy based on the use of \textbf{Poincar\'e inequalities}, as well as on {\bf optimal transport techniques}.  See \cite{Bonis}  for a further amplification of the approach of \cite{l-n-p-15}, with applications to the quantitative multidimensional CLT in the 2-Wasserstein distance. See also \cite{D20}.

\begin{Thm}\label{Thm:FI}
Let	$d\nu = h d \gamma$ be a centered probability measure on $\R^d$ with smooth density function $h$ with respect to the standard Gaussian measure $\gamma$.
	\begin{enumerate}
\item[\rm (1)] Then the following {\bf Gaussian $\HSI$ inequality} holds: 
$$ \HHH(\nu,\gamma) \le \frac{1}{2} \SSS^2(\nu,\gamma) \log \Big(   1+ \frac{\III(\nu,\gamma)}{\SSS^2(\nu,\gamma)} \Big).$$
\item[\rm (2)] Assume further that  $   \SSS(\nu,\gamma)    $ and $\HHH(\nu,\gamma )$ are both positive and finite. Then, the following {\bf Gaussian $\WSH$ inequality} holds: 
$$ \W_2 (\nu,\gamma) \le \SSS(\nu,\gamma) \arccos \left(  e^{- \frac{\HHH(\nu,\gamma)}{\SSS^2(\nu,\gamma)}}  \right).$$
	\end{enumerate}
\end{Thm}

 The next subsection deals with the challenging problem of quantitative probabilistic approximations in infinite dimension.

{
\subsection{Functional approximations and Dirichlet structures}

Although Stein's method is already successfully used for quantifying functional limit theorems of the Donsker-type (see \cite{Barbour-PTRF, BJ09}, as well as \cite{DK21, DKP19, K20, Shih} for a discussion of recent developments), the general problem of assessing the discrepancy between probability distributions on infinite-dimensional spaces (like e.g. on classes of smooth functions or on the Skorohod space) is essentially open. 

 In the last years a new direction of research has emerged, where the ideas behind the Malliavin-Stein approach are applied in the framework of {\bf Dirichlet structures}, in order to deal with quantitative estimates on the probabilistic approximation of Hilbert space-valued random variables. A general (and impressive!) contribution on the matter is the recent work by Bourguin and Campese \cite{BC20}, where the authors are able to retrieve several Hilbert space counterparts of the finite-dimensional results discussed in Section \ref{s:malliavin} above. Bourguin and Campese's approach (whose discussion requires preliminaries that go beyond the scope of our survey) represents a substantial addition to a line of investigation intiated by L. Coutin and L. Decreusefond in the seminal works  \cite{CD13, Laurent-Donsker-Wass,Laurent-Higher-Order-Stein,Laurent-stein-roughPath,Laurant-Stein-Dirichlet-Malliavin-method}.

 As a quick illustration, we conclude the section with two representative statements, taken from \cite{CD13, Laurant-Stein-Dirichlet-Malliavin-method} and \cite{Laurent-Donsker-Wass}, respectively.

\begin{Thm}[See \cite{CD13} and Section 3.2 in \cite{Laurant-Stein-Dirichlet-Malliavin-method}]
Let $(N_\lambda (t)  :  t \ge 0)$ be a Poisson process with intensity $\lambda$.  Then,  as $\lambda \to \infty$,  $$  \left( \frac{N_\lambda(t) - \lambda t}{\sqrt{\lambda}} :  t\ge 0 \right)  \implies  \left(  B(t)  : t \ge 0 \right) $$
where the convergence takes place weakly in the Skorohkod space. Moreover, for every $\beta < \frac{1}{2}$  consider the so-called Besov-Liouville space $I_{\beta,2}$,

\begin{equation*}
I_{\beta,2}  = \Big\{     f   \, : \,  \exists \, \dot{f},  \, f(x) = \frac{1}{\Gamma(\beta)}  \int_{0}^{x} (x-t)^{\beta -1} \dot{f}(t)  dt            \Big \}.
\end{equation*}
Let $\mu_\beta$ denote the Wiener measure on the space $I_{\beta,2}$, and $Q_\lambda$ is the probability measure induced by $\left(  N_\lambda(t) : t \ge 0 \right)$ seen as canonical process is Poisson process with intensity $\lambda$.  Then, there exists a constant $c_\beta$ such that

\begin{equation*}
\sup_{\Vert F \Vert_{   C^2_b (I_{\beta,2},\R)  } \le 1}   \Big \vert  \int  F dQ_\lambda -  \int   F  d \mu_\beta \Big \vert   \le  \frac{c_\beta}{\sqrt{\lambda}}
\end{equation*}
where  $ C^2_b (I_{\beta,2},\R)$  is the set of twice Fr\'echet differentiable functionals on $I^{\beta,2}$.
 
\end{Thm}

The next result aims to provide a rate of convergence in Donsker theorem in Wasserstein distance. Let $\eta \in (0,1)$, $p \ge 1$. Define the fractional Sobolev space $W_{\eta,p}$ as the closure of the space $C^1$ w.r.t norm

\begin{equation*}
\Vert f\Vert^p_{\eta,p} := \int_{0}^{1}  \vert f (t) \vert^p dt + \int_{0}^{1} \int_{0}^{1}   \frac{    \vert f(t) - f(s) \vert^{p}}{  \vert    t -s \vert^{1+p\eta}} ds dt.  
\end{equation*}  

Also, for $n \ge 1$, define $\mathcal{A}^n=  \{  (k,j) \, :\,  1 \le k \le d,  \, 0 \le j \le n-1   \}$, and let 

\begin{equation*}
S^n = \sum_{(k,j)\in \mathcal{A}^n} X_{(k,j)} h^n_{(k,j)}, \quad  h^n_{(k,j)} (t) = \sqrt{n} \int_{0}^{t} \textbf{1}_{[j/n, (j+1)/n]}(s) ds \, e_{k}
\end{equation*}
where $(e_k) : 1 \le k \le d$ is the canonical basis of $\R^d$, and $(X_{(k,j)}, (k,j) \in \mathcal{A}^n)$ is a family of independent identically distributed, $\R^d$-valued, random variables with $\E[X]=0$, and $\E \Vert X \Vert^2_{\R^d}   =1$, where $X$ is a random variable which has their common distribution.

\begin{Thm}[See Section 3 in \cite{Laurent-Donsker-Wass}]
Let $W=W_{\eta,p}\left( [0,1], \R^d  \right)$, and $\mu_{\eta,p}$ be the law of the $d$-dimensional Brownian motion $B$ on the space $W$.  Then, there exists a constant $c$ such that for $X \in L^p  (W;\R^d,\mu_{\eta,p})$ with $p \ge 3$,

\begin{equation*}
\sup_{ F  \in  \text{Lip}_1 (W_{\eta,p})}  \Big \vert    \E[F(S^n)]  -  \E[F(B)]         \Big \vert  \le  c \, \Vert X\Vert ^p_{L^p}  \, n^{-\frac{1}{6} + \frac{\eta}{3}}  \ln n
\end{equation*}
where 
\begin{equation*}
\text{Lip}_1 (W_{\eta,p})   :  =  \Big \{       F:W_{\eta,p}  \to \R^d \, :  \,  \Vert F(x)  - F(y) \Vert_{\R^d}  \le \Vert  x -y \Vert_{\eta,p},  \, \forall x,y \in W_{\eta,p}           \Big \}.
\end{equation*}
\end{Thm}

}

\medskip

Further applications of Malliavin-Stein techniques in the framework of Dirichlet structures are contained in \cite{DST, DH}. The next section focuses on a discrete Markov structure for which exact fourth moment estimates are available.

\section{{Bounds on the Poisson space: fourth moments, second-order Poincar\'e estimates and two-scale stabilization}}

We will now describe a non-diffusive Markov triple for which a fourth moment result analogous to Proposition \ref{prop:general-setting} holds. Such a Markov triple is associated with the space of square-integrable functionals of a Poisson measure on a general pair $(Z, \mathscr{Z})$, where $Z$ is a Polish space and $\mathscr{Z}$ is the associated Borel $\sigma$-field. The requirement that $Z$ is Polish -- together with several other assumptions adopted in the present section -- is made in order to simplify the discussion; the reader is referred to \cite{dp, dvz} for statements and proofs in the most general setting. See also \cite{Lsurvey, LPbook} for an exhaustive presentation of tools of stochastic analysis for  functionals of Poisson processes, as well as \cite{PRbook} for a discussion of the relevance of variational techniques in the framework of modern stochastic geometry.

\subsection{Setup}

Let $\mu$ be a non-atomic $\sigma$-finite measure on $(Z, \mathscr{Z})$, and set $\mathscr{Z}_\mu:=\{B\in\mathscr{Z}\,:\,\mu(B)<\infty\}$. In what follows, we will denote by
\begin{equation*}
\eta=\{\eta(B)\,:\,B\in\mathscr{Z}\}
\end{equation*}
a \textbf{Poisson measure} on $(Z,\mathscr{Z})$ with \textbf{control} (or {\bf intensity}) $\mu$. This means that $\eta$ is a random field indexed by the elements of $\mathscr{Z}$, satisfying the following two properties: (i) for every finite collection $B_1,\dotsc,B_m\in\mathscr{Z}$ of pairwise disjoint sets, the random variables $\eta(B_1),\dotsc,\eta(B_m)$ are stochastically independent, and (ii) for every $B\in\mathscr{Z}$, the random variable $\eta(B)$ has the Poisson distribution with mean $\mu(B)$ \footnote{Here, we adopt the usual convention of identifying a Poisson random variable with mean zero (resp. with infinite mean) with an a.s. zero (resp. infinite) random variable}. Whenever $B\in\mathscr{Z}_\mu$, we also write 
$\hat{\eta}(B):=\eta(B)-\mu(B)$ and denote by 
\[\hat{\eta}=\{\hat{\eta}(B)\,:\,B\in\mathscr{Z}_\mu\}\]
the \textbf{compensated Poisson measure} associated with $\eta$. Throughout this section, we assume that $\mathscr{F} = \sigma(\eta)$.

\medskip

It is a well-known fact that one can regard  the Poisson measure $\eta$ as a random element {{} taking values} in the space $\mathbf{N}_\sigma=\mathbf{N}_\sigma({Z})$ of all $\sigma$-finite point measures $\chi$ on $({Z},\mathscr{Z})$ that satisfy $\chi(B)\in\N_0\cup\{+\infty\}$ for all $B\in\mathscr{Z}$. {{} Such a space is} equipped with the smallest 
$\sigma$-field $\mathscr{N}_\sigma:=\mathscr{N}_\sigma(Z)$ {{} such that}, for each $B\in\mathscr{Z}$,  the mapping $\mathbf{N}_\sigma\ni\chi\mapsto\chi(B)\in[0,+\infty]$ is measurable. {In view of our assumptions on $Z$ and following e.g. \cite[Section 6.1]{LPbook}, throughout the paper we can assume without loss of generality that $\eta$ is {\bf proper}, in the sense that $\eta$ can be $P$-a.s. represented in the form
\begin{equation}\label{e:proper}
\eta = \sum_{n=1}^{\eta({Z})}\delta_{X_n},
\end{equation}
where $\{X_n : n\geq 1\}$ is a countable collection of random elements with values in $\mathcal{Z}$ and where we write $\delta_z$ for the \textit{Dirac measure} at $z$. Since we assume $\mu$ to be non-atomic, one has that $X_k\neq X_n$ for every $k\neq n$, $P$-a.s..

\medskip

Now denote by $\mathbf{F}(\mathbf{N}_\sigma)$ the class of all measurable functions $\mathfrak{f}:\mathbf{N}_\sigma\rightarrow\R$ and by $\mathcal{L}^0(\Omega):=\mathcal{L}^0(\Omega,\mathscr{F})$ the class of real-valued, measurable functions $F$ on $\Omega$. 
Note that, as $\mathscr{F}=\sigma(\eta)$, each $F\in \mathcal{L}^0(\Omega)$ has the form $F=\mfk(\eta)$ for some measurable function $\mfk$. This $\mfk$, called a {\bf representative} of $F$, is $P_\eta$-a.s. uniquely defined, where $P_\eta=P\circ\eta^{-1}$ is the image measure of $P$ under $\eta$. Using a representative $\mfk$ of $F$, one can introduce the \textbf{add-one cost operator} $D^+=(D_z^+)_{z\in\mathcal{Z}}$ on $\mathcal{L}^0(\Omega)$ as follows: 
\begin{equation}\label{defdp}
 D_z^+F:=\mfk(\eta+\delta_z)-\mfk(\eta)\,,\quad z\in\mathcal{Z};
\end{equation}
similarly, we define $D^-$ on $\mathcal{L}^0(\Omega)$ as
\begin{equation}\label{defdm}
 D_z^-F:=\mfk(\eta)-\mfk(\eta-\delta_z)\,,\,\, \mbox{ if\ \ } z\in {\rm supp}({\eta}) \,,\,\, \mbox{and\ \ } D_z^-F:=0, \,\, \mbox{otherwise,} \end{equation}
{where 
$ {\rm supp}({\chi}):=\bigl\{z\in\mathcal{Z}\,:\, \text{for all $A\in\mathscr{Z}$ s.t. $z\in A$: $\chi(A)\geq1$}\bigr\}$
is the support of the measure $\chi\in{\bf N}_\sigma$}. We call $-D^-$ the \textbf{remove-one cost operator} associated with $\eta$. We stress that the definitions of $D^+F$ and $D^-F$ are, respectively, $P\otimes \mu$-a.e. and $P$-a.s. independent of the choice of the representative $\mfk$ --- see e.g. the discussion contained \cite[Section 2]{dp} and the references therein. Note that he operator $D^+$ can be straightforwardly iterated by setting $D^{(1)}:=D^{+}$ and, for $n\geq2$ and $z_1,\dotsc,z_n\in Z$ and $F\in\mathcal{L}^0(\Omega)$, by recursively defining
\begin{equation*}
 D^{(n)}_{z_1,\dotsc,z_n}F:=D^{+}_{z_1}\bigl(D^{(n-1)}_{z_2,\dotsc,z_n}F\bigr).
\end{equation*}

\subsection{$L^1$ integration by parts}

One of the most fundamental formulae in the theory of Poisson processes is the so-called \textbf{Mecke formula} stating that, for each measurable function $h:\mathbf{N}_\sigma\times Z\rightarrow [0,+\infty]$, the identity
\begin{equation}\label{mecke}
\E\biggl[\int_Z h(\eta+\delta_z,z)\mu(dz)\biggr] = \E\biggl[\int_Z h(\eta,z)\eta(dz)\biggr]
\end{equation}
holds true; see \cite[Chapter 4]{LPbook} for a detailed discussion. Such a formula can be used in order to define an (approximate) integration by parts formula on the Poisson space, 

\medskip

{{} For random variables $F,G\in \mathcal{L}^0(\Omega)$ such that $D^+F\, D^+G\in L^1(P \otimes \mu)$, we define
\begin{equation}\label{e:Gamma0}
\Gamma_0(F,G) := \frac12\left\{  \int_Z (D_z^+F D_z^+G) \, \mu(dz)  + \int_Z (D_z^-F D_z^-G) \, \eta(dz)\right\}
\end{equation} 
which verifies $\E[|\Gamma_0(F,G) |] <\infty$, and $\E[\Gamma_0(F,G)] =\E[  \int_Z (D_z^+F D_z^+G) \,  \mu(dz) ]$, in view of Mecke formula. The following statement, taken from \cite{dp}, can be regarded as an integration by parts formula in the framework of Poisson random measures, playing a role similar to that of Lemma \ref{L : Tech1} in the setting of Gaussian fields. It is an almost direct consequence of \eqref{mecke}.

\begin{Lem}[$L^1$ integration by parts]\label{l:ibpl1} Let $G,H \in \mathcal{L}^0(\Omega)$ be such that $$G D^+H, \,\, D^+G\, D^+H\in L^1(P\otimes \mu).$$  
Then,
\begin{equation}\label{e:ibpl1}
\E\left[ G\left(\int_Z D_z^+H\, \mu(dz) - \int_Z D_z^ -H\, \eta(dz)\right)\right] = -\E[\Gamma_0(G,H)]. 
\end{equation}
\end{Lem}

\medskip

We will now focus on multiple Wiener-It\^o integrals.

\subsection{Multiple integrals}

 For an integer $p\geq1$ we denote by $L^2(\mu^p)$ the Hilbert space of all square-integrable and real-valued functions on $\mathcal{Z}^p$ and we write $L^2_s(\mu^p)$ 
for the subspace of those functions in $L^2(\mu^p)$ which are $\mu^p$-a.e. symmetric. Moreover, for ease of notation, we denote by $\Enorm{\cdot}$ and $\langle \cdot,\cdot\rangle_2$ the usual norm and scalar product 
on $L^2(\mu^p)$ for whatever value of $p$. We further define $L^2(\mu^0):=\R$. For $f\in L^2(\mu^p)$, we denote by $I_p(f)$ the \textbf{multiple Wiener-It\^o integral} of $f$ with respect to $\hat{\eta}$. 
If $p=0$, then, by convention, 
$I_0(c):=c$ for each $c\in\R$. Now let $p,q\geq0$ be integers. The following basic properties are proved e.g. in \cite{Lsurvey}, and are analogous to the properties of multiple integrals in a Gaussian framework, as discussed in Section \ref{ss:isonormal}: 
\begin{enumerate}
 
 \item $I_p(f)=I_p(\tilde{f})$, where $\tilde{f}$ denotes the \textbf{canonical symmetrization} of $f\in L^2(\mu^p)$;
 
 \item $I_p(f)\in L^2(P)$, and $\E\bigl[I_p(f)I_q(g)\bigr]= \delta_{p,q}\,p!\,\langle \tilde{f},\tilde{g}\rangle_2 $, where $\delta_{p,q}$ denotes the Kronecker's delta symbol.
\end{enumerate}

\smallskip

As in the Gaussian framework of Section \ref{ss:isonormal}, for $p\geq0$ the Hilbert space consisting of all random variables $I_p(f)$, $f\in L^2(\mu^p)$, is called the \textbf{$p$-th Wiener chaos} associated with $\eta$, and is customarily denoted by $C_p$. It is a crucial fact that every $F\in L^2(P)$ admits a unique representation 
\begin{equation}\label{chaosdec}
 F=\E[F]+\sum_{p=1}^\infty I_p(f_p)\,,
\end{equation}
where $f_p\in L_s^2(\mu^p)$, $p\geq1$, are suitable symmetric kernel functions, and the series converges in $L^2(P)$. Identity \eqref{chaosdec} is the analogous of relation \eqref{E}, and is once again referred to as the \textbf{chaotic decomposition} of the functional $F\in L^2(P)$. 

\medskip

The multiple integrals discussed in this section also enjoy multiplicative properties similar to formula \eqref{multiplication} above -- see e.g. \cite[Proposition 5]{Lsurvey} for a precise statement. One consequence of such product formulae is that, if $F\in C_p$ and $G\in C_q$ are such that $FG$ is square-integrable, then
\begin{equation}\label{e:fg}
FG \in \bigoplus_{r=0}^{p+q}C_r,
\end{equation}
which cab be seen as a property analogous to \eqref{fundamental-assumtionbis}.
\subsection{Malliavin operators}

 We now briefly discuss Malliavin operators on the Poisson space. 
\begin{enumerate}

\item The \textbf{domain} $\dom D$ of the {\bf Malliavin derivative operator }$D$ is the set of all $F\in L^2( P )$ such that the chaotic decomposition \eqref{chaosdec} of $F$ satisfies $\sum_{p=1}^\infty p\,p!\Enorm{f_p}^2<\infty$. For such an $F$, the random function $Z \ni z\mapsto D_zF\in L^2( P )$ is defined via
\begin{equation}\label{e:D}
 D_zF=\sum_{p=1}^\infty p I_{p-1}\bigl(f_p(z,\cdot)\bigr)\,,
\end{equation}
{whenever $z$ is such that the series is converging in $L^2( P )$ (this happens a.e.-$\mu$), and set to zero otherwise; note that $f_p(z,\cdot)$ is an a.e. symmetric function on ${Z}^{p-1}$.} Hence, $DF=(D_zF)_{z\in\mathcal{Z}}$ is indeed an element of $L^2\bigl( P \otimes \mu\bigr)$. 
 It is well-known that $F \in \dom D$ if and only if $D^+F\in L^2\bigl( P \otimes \mu\bigr)$, and in this case 
\begin{equation}\label{altDz}
D_zF=D_z^+F,\quad  P \otimes \mu{\rm -a.e.} . 
\end{equation}

\item The domain $\dom {\bf L}$ of the \textbf{Ornstein-Uhlenbeck generator} ${\bf L}$ is the set of those $F\in L^2( P )$ whose chaotic decomposition \eqref{chaosdec} verifies $\sum_{p=1}^\infty p^2\,p!\Enorm{f_p}^2<\infty$ (so that $\dom {\bf L} \subset \dom D$) and, for $F\in\dom {\bf L}$, one defines
\begin{equation}\label{defL}
 {\bf L} F=-\sum_{p=1}^\infty p I_p(f_p)\,.
\end{equation}
By definition, $\E[{\bf L} F]=0$; also, from \eqref{defL} it is easy to see that ${\bf L} $ is \textbf{symmetric}, in the sense that 
\begin{equation*}
 \E\bigl[({\bf L} F)G\bigr]=\E\bigl[F({\bf L} G)\bigr]
\end{equation*}
for all $F,G\in\dom{\bf L} $. Note that, from \eqref{defL}, it is immediate that the spectrum of $-{\bf L}$ is given by the nonnegative integers and that $F\in \dom {\bf L} $ is an eigenfunction of $-{\bf L}$ with corresponding eigenvalue $p$ if and only if $F=I_p(f_p)$ for some $f_p\in L^2_s(\mu^p)$, that is: 
\begin{equation*}
 C_p={\rm Ker}({\bf L}+pI).
\end{equation*}
{} The following identity corresponds to formula (65) in \cite{Lsurvey}: if $F\in\dom {\bf L}$ is such that $D^+ F\in L^1( P \otimes\mu)$, then 
\begin{equation}\label{formL}
 {\bf L}F=\int_ \mathcal{Z} \bigl(D_z^+F\bigr) \mu(dz)-\int_\mathcal{Z}\bigl(D_z^{-}F\bigr)\eta(dz)\,.
\end{equation}

{Define for any $F\in L^2(P)$ the pseudo-inverse ${\bf L}^{-1}$ by
\begin{align*}
{\bf L}^{-1} F = - \sum_{p=1}^\infty \frac{1}{p} I_p(f_p). 
\end{align*}
Recall the covariance identity
\begin{align}\label{e:covariance}
\Cov( F, G) = - \int  \E[ D_zG D_z {\bf L}^{-1}F] \mu(dz)
\end{align}}


\item For suitable random variables $F,G\in\dom {\bf L}$ such that $FG\in\dom {\bf L}$, we introduce the \textbf{carr\'{e} du champ operator} $\Gamma$ associated with ${\bf L}$ by 
\begin{equation}\label{cdc}
 \Gamma(F,G):=\frac{1}{2}\bigl({\bf L}(FG)-F{\bf L}G-G{\bf L}F\bigr)\,.
\end{equation}
{{} The symmetry of ${\bf L}$} implies immediately the crucial \textbf{integration by parts formula} 
\begin{equation}\label{intparts}
 \E\bigl[({\bf L}F)G\bigr]=\E\bigl[F({\bf L}G)\bigr]=-\E\bigl[\Gamma(F,G)\bigr];
\end{equation}
we will see below that, for many random variables $F,G$, relation \eqref{intparts} is indeed the same as identity appearing in Lemma \ref{l:ibpl1}. 

\smallskip

\end{enumerate}

The following result -- proved in \cite{dp} -- provides an explicit representation of the carr\'{e}-du-champ operator $\Gamma$ in terms of $\Gamma_0$, as introduced in \eqref{e:Gamma0}. 
\begin{prop}\label{cdcform}
 For all $F,G\in\dom {\bf L}$ such that $FG\in\dom {\bf L}$ and $$DF,\, DG,\, FDG, \,GDF \in L^1( P \otimes\mu),$$ we have that $DF= D^+ F, \, DG= D^+G$, in such a way that $DF\, DG = D^+F\, D^+G  \in L^1( P \otimes\mu)$, and
 \begin{equation}\label{e:identityGamma}
  \Gamma(F,G)=\Gamma_0(F,G),
 \end{equation}
where $\Gamma_0$ is defined in \eqref{e:Gamma0}.
\end{prop}

One crucial consequence of such a result is that the operator $\Gamma$ {\it is not diffusive}, in such a way that the triple $(\Omega, P, {\bf L})$ {\it is not a diffusive fourth moment structure}, such as the ones introduced in Definition \ref{d:dfms}; it follows in particular that the machinery of Section \ref{s:mt} cannot be directly applied.

\subsection{Fourth moment theorems}
Starting at least from the reference \cite{pstu} (where Malliavin calculus and Stein's method were first combined on the Poisson space), it has been an open problem for several years that of establishing a fourth moment bound similar to Theorem \ref{Thm:4mt-general} on the Poisson space. As recalled above, the main difficulty in achieving such a result is the discrete nature of add-one and remove-one cost operators, preventing in particular the triple $(\Omega, P, {\bf L})$ from enjoying a diffusive property.

\medskip

The next statement contains one of the main bounds proved in \cite{dvz}, and shows that a quantitative fourth moment bound is available on the Poisson space. Such a bound (which also has a multidimensional extension) is proved by a clever combination of Malliavin-type techniques with an infinitesimal version of the {\bf exchangeable pairs approach} toward Stein's method -- see e.g. \cite{Chen-book}. 

\begin{Thm} For $q\geq 2$, let $F = I_q(f_q)$ be a multiple integral of order $q$ with respect to $\hat{\eta}$, and assume that $\E[F^2]=1$. Then,
$$
d_W(F,N)\leq \left(  \sqrt{\frac{2}{\pi}} +\frac43   \right)\sqrt{\E[F^4] - 3}.
$$
\end{Thm}

One should notice that the first bound of this type was proved in \cite{dp} under slightly more restrictive assumptions; also, reference \cite{dp} contains analogous bounds in the Kolmogorov distance, that are not achievable by using exchangeable pairs. In particular, one of the key estimates used in \cite{dp} is the following remarkable equality and bound
\begin{equation*}
\frac{1}{2q}\int_Z\E\bigl[\abs{D_z^+F}^4\bigr]\mu(dz)=\frac{3}{q}\E\bigl[F^2\Gamma(F,F)\bigr]-\E\bigl[F^4\bigr]\leq \frac{4q-3}{2q}\Bigl(\E\bigl[F^4\bigr]-3\E[F^2]^2\Bigr),
\end{equation*}
that are valid for every $F\in C_q$, $q\geq 2$, such that the mapping $z\mapsto D_z^+F$ verifies some minimal integrability conditions. 

{
\subsection{Second-order Poincar\'e estimates}

What one calls {\bf second-order Poincar\'e inequalities} is a collection of analytic estimates (first established on the Poisson space in \cite{LPS16}) where the Wasserstein and Kolmogorov distances, between a given function of $\eta$ and a Gaussian random variable, are bounded by integrated moments of iterated add-one-cost operators on the Poisson space. The ratio behind such a name is the following. Just as the Poincar\'e inequality 
\begin{align}\label{e:Poincare}
\Var(F) \le \int_Z \E[(D^+_z F)^2]  \mu(dz),
\end{align}
controls the variance of a random variable $F$ by means of integrated moments of the add-one cost (see \cite[Section 18.3]{LPbook}), the discrepancy between the distribution of $F$ and that of a Gaussian random variable is controlled by integrated moments of second-order add-one-cost $D^+_xD^+_y F := D^2_{z,y}F$, a phenomenon already observed in the Gaussian setting \cite{cha, NPR09, v}, where gradients 
typically replace add-one-cost operators. 

For the rest of the section, we exclusively consider square-integrable random variables $F$ such that $F\in {\rm dom}\, D$, in such a way that $D^+F = DF$ (up to negligible sets). The starting point for proving second-order Poincar\'e estimates is the covariance identity \eqref{e:covariance}, which can be proved as in the Gaussian setting by means of chaos expansions. When one combines Stein's method with such a formula, it is however not possible to deduce the existence of a Stein kernel as in the Gaussian setting (see \eqref{e:ibpg}), since Malliavin operators on a Poisson space {\it do not} enjoy an exact chain rule such as \eqref{e:chainrule}. Indeed, we have that, for sufficiently smooth mapping $f:\R\to \R$,
\begin{align*}
\Cov(F, f(F)) &= -\int \E[ D_z(f(F)) D_z {\bf L}^{-1} F ] \mu(dz) \\
 &= - \int \E[ f'(F) D_z F D_z {\bf L}^{-1} F] \mu(dz) + R
\end{align*} 
where we approximate $D_z(f(F))=f(F+D_z F) -f(F)$ by $f'(F) D_z F$ with the error term
\begin{align*}
D_z F  \int_0^1 [ f'(F+ tD_zF) -f'(F)] dt
\end{align*}
appearing in the implicit definition of $R$; notice that, in general, one has that \xy{$R\neq 0$}, in such a way that the previous computations do not yield the existence of a Stein kernel. Selecting $f$ as in Lemma \ref{lem:stein}-(d), one can bound the error term in the aforementioned calculation by $|D_z F|^2$. Therefore, for $F$ such that $\E[F]=0$ and $\Var[F]=1$, one has the bound
\begin{align*}
d_W(F,N) \le \sqrt{ \Var\Big[ \int D_z F D_z {\bf L}^{-1} F] \mu(dz)\Big] } + \int  \E[ |D_z F|^2 |D_z {\bf L}^{-1}F| ] \mu(dz).
\end{align*}
Applying the Poincar\'e inequality \eqref{e:Poincare} to the variance term, as well as the contraction bound \cite[Lemma 3.4]{LPS16} for the add-one-cost
\begin{align*}
\E[|D_z {\bf L}^{-1} F|^p] \le \E[|D_z F|^p ], \quad p\ge 1,
\end{align*}
and analogous estimates for the iterated add-one-cost, leads to the following theorem.
\begin{Thm}[Second-order Poincar\'e estimates  \cite{LPS16}]
Let $F\in \dom D$ be such that $\E[F]=0$ and $\Var[F]=1$, and let $N$ be a standard Gaussian random variable. Then,
\begin{align*}
d_W(F,N)\le \gamma_1 + \gamma_2 + \gamma_3,
\end{align*}
where
\begin{align*}
\gamma_1 &:= 2\Big[ \iiint  \E[ (D_x F D_y F)^2]^{1/2} \E[ (D^2_{x,z} F D^2_{y,z} F)^2 ]^{1/2}  \mu^3(dxdydz)  \Big]^{1/2},\\
\gamma_2&:= \Big[\iiint \E[(D^2_{x,z} F  D^2_{y,z}F)^2] \mu^3(dxdydz) \Big]^{1/2},\\
\gamma_3&:= \int \E[|D_x F|^3] \mu(dx).
\end{align*}
\end{Thm}

As mentioned above, second-order Poincar\'e techniques are equally useful for obtaining bounds in the Kolmogorov distance -- see \cite{LPS16}, as well as \cite{SY19} for a powerful extension to the framework of multivariate normal approximations.   \\

An example of a successful application of second-order Poincar\'e estimates from \cite{LPS16} (to which we refer the reader for a discussion of the associated literature) is the derivation of presumably optimal Berry-Esseen bounds for the total edge length of the Poisson-based nearest neighbor graph. More precisely, let $\eta_t$ be a Poisson point process with intensity $t>0$ on a convex compact set $H\subset \R^d$.  We consider the graph with vertex set $\supp \eta_t$ and edge set formed by $\{x,y\}\subset \supp \eta_t$ when either $x$ is the nearest neighbor of $y$ or the other way around. Consider the total edge length of the graph so obtained, denoted by $L_t$. Then we have
\begin{align*}
d_W\Big(\frac{L_t - \E[L_t]}{\sqrt{\Var[L_t]}},N\Big), \, d_K\Big(\frac{L_t - \E[L_t]}{\sqrt{\Var[L_t]}},N\Big) \le \frac{C}{\sqrt{t}},
\end{align*}
where $d_K$ is the one-dimensional Kolmogorov distance, and $C$ depends only on $H$. We refer the reader to \cite[Theorem 7.1]{LPS16} for a far more general statement, and to \cite{LrSY19} for a collection of presumably optimal bounds on the normal approximation of {\bf exponentially stabilizing} random variables (see the next subsection). }

{
\subsection{Stabilization theory and two-scale bounds}

While the second-order Poincar\'e estimates can provide sharp Berry-Esseen bounds, they are not always applicable. This is the case, for instance, for certain combinatorial optimisation statistics or connectivity functionals of the underlying Poisson process. The problem is typically that the iterated add-one-cost of the functionals, although well-defined almost surely, are not computationally tractable, e.g. for obtaining moment estimates. 

In this section, we present an alternative collection of analytic inequalities, called the {\bf two-scale stabilization bounds}, which avoid the use of  iterated add-one-cost -- they are one of the main findings from \cite{LrPY20+}; see also \cite{CS17} for several related estimates obtained by a discretization procedure. As their name suggests,  these bounds are closely related to the stabilization theory of Penrose and Yukich \cite{PY01, Penrose05}. Such a theory originated from the ground-breaking central limit theorem of Kesten and Lee \cite{KL97} for the total edge weight $M_n$ of Euclidean minimal spanning trees (MST) with stationary Poisson points $\eta_n$ in a ball of radius $n\in\mathbb{N}$. Recall that the MST is the connected graph over the vertex set $\eta_n$ that minimises its total length.  Without referring to the stochastic analysis on the Poisson space, Kesten and Lee already performed a fine study of the add-one-cost of $M_n$ (and {\it not} of the iterated add-one-cost) implying some moment estimates of $D_x M_n$. Penrose and Yukich \cite{PY01} extrapolated the high level ideas from \cite{KL97} and transformed them into a general theory applicable to (non-quantitative) central limit theorems for a plethora of problems in stochastic geometry. 
The theory was further extended to multivariate normal approximation by Penrose \cite{Penrose05}. A variant of the theory using score functionals was put forward by Baryshnikov and Yukich \cite{BY05}.

We now define properly the notions of strong and weak stabilization. We assume for concreteness that the ambient space is $\R^d$ and $\eta$ is a Poisson process of unit intensity. A Poisson functional $F=F(\eta)$ is  {\bf strongly stabilizing} if there exists an almost surely finite random variable $R$, called the {\bf stabilization radius}, such that
\begin{align*}
D_0 F( \eta|_{B_R}) = D_0 F(\eta),
\end{align*}
where $B_R$ stands for a ball with radius $R$ centered at the origin. Here is a simple example. Fix $r>0$ and make an edge between two points in $\eta_n:=\eta|_{B_n}$ within distance $r$. The graph $G(\eta_n,r)$ so obtained is known as the {\bf Gilbert graph} or the {\bf random geometric graph}. Then, the number $F(\eta_n)$ of edges within a finite window containing the origin has stabilization radius $R=r$ almost surely, since $D_0 F(\eta)$ is the number of edges incident to the origin in $G(\eta+\delta_0,r)$. Proving strong stabilization often relies on combinatorial and geometric arguments in many problems of stochastic geometry, see \cite{PY01} for a list of  examples. In general situations, $R$ is genuinely random in contrast to the simple example given above. 

To obtain central limit theorems, it actually suffices to show a weaker version of stabilization. We say that $F$ is {\bf weakly stabilizing} if for \emph{any} sequence of measurable sets $E_n$ satisfying $\liminf E_n=\R^d$, we have the almost sure convergence
\begin{align*}
D_0 F(\eta|_{E_n}) \to \Delta
\end{align*}
where $\Delta$ is a random variable. 
It is clear that a strongly stabilizing functional is also weakly stabilizing with $\Delta=D_0F(\eta)$. 

\begin{Thm}[ See{\cite[Theorem 3.1]{PY01}}]\label{t:PY} Suppose that $F$ is weakly stabilizing and satisfies the moment condition
\begin{align*}
\sup_A \E[ |D_0 F(\eta|_A) |^4] <\infty,
\end{align*}
where the supremum is taken for all ball $A$ that contains $0$. Then there exists $\sigma^2\ge 0$, such that
\begin{align*}
\frac{1}{\sqrt{\mathrm{Vol}(B_n)}}(F(\eta_n) - \E[F(\eta_n)]) \overset{d}\to N(0,\sigma^2).
\end{align*}
\end{Thm}

It is remarkable how few assumptions one needs in order to obtain a CLT, which somehow indicates that stabilization is the right condition for Gaussian approximations on the Poisson space. Notice that the limiting variance $\sigma^2$ could be 0. In \cite{PY01}, it was shown that $\sigma^2>0$ whenever $\Delta$ is not a constant. 
Theorem \ref{t:PY} was proved by a martingale method and does not offer insights on how fast the normalized sequence converges to normal. The latter question was addressed by a recent preprint by Lachi\`eze-Rey, Peccati and Yang \cite{LrPY20+}. Under slightly strengthened  conditions on the functionals, they assessed the rate of normal approximation in Theorem \ref{t:PY}. To state one of the bound that can be deduced from \cite{LrPY20+}, we consider again the ball $B_n$ of radius $n$ centered at the origin, and introduce the key quantity
\begin{align*}
\psi_n:= \sup_{x\in B_n} \E[|D_x F(\eta|_{B_n}) - D_x F(\eta|_{A_{n,x}})|]., \quad n\geq 1.
\end{align*}
In practice, we take $A_{n,x}=B_{b_n}(x)=\{y: |x-y|\le b_n\}$ with  $1\ll b_n\ll n$ which is a local window of $x$ compared to the scale of $B_n$. In what follows, we make this choice and call $\psi_n$ a {\bf two-scale discrepancy} in view of this interpretation.
The following result, taken from \cite{LrPY20+}, can be applied in many concrete applications. 

\begin{Thm}[{\cite[Corollary 1.3]{LrPY20+}}] \label{t:LrPY} Set $\hat F_n = (F(\eta_n)-\E[F(\eta_n)])/\sqrt{\Var[F(\eta_n)]}$, where we have set $\eta_n=\eta|_{B_n}$ as before.  Suppose that
\begin{align*}
\sup_{n\in\mathbb{N}, x\in B_n}\E[|D_x F(\eta_n)|^p]<\infty
\end{align*}
for some $p>4$ and also that there exists an absolute constant $b>0$ such that $\Var[F(\eta_n)]\ge b|B_n|$. Then there exists a finite positive constant $c$ such that
\begin{align*}
\frac{1}{c} d_W(\hat F_n, N(0,1))\le  \psi_n^{\frac{1}{2}(1-\frac{4}{p})} + \Big(\frac{b_n}{n}\Big)^{\frac{d}{2}}.
\end{align*}
\end{Thm}

This theorem simplifies and extends some arguments in the proof of a quantitative CLT for the minimal spanning trees by Chatterjee and Sen \cite{CS17}.  
Analogous Kolmogorov bounds for univariate normal approximation, and bounds for multivariate normal approximation are also considered in \cite{LrPY20+}.  More remarks are in order.

\begin{Rem}{\rm
\begin{itemize}
\item[i)]The sequence $(b_n)$ serves as a free parameter in the bound. One should keep track of the dependence of $\psi_n$ on $b_n$ and make an optimization in the end.  
\item[ii)] For any fixed $x\in\R^d$, applying the weak stabilization condition for $F$ with two sequences $(B_n)$ and $(B_{b_n}(x))$ (together with the translation invariance of $\eta$ and the  moment assumption for the add-one-cost) yields the following convergence
\begin{align*}
\E[ | D_x F(\eta|_{B_n}) - D_x F(\eta|_{B_{b_n}(x)}) | ] \to 0.
\end{align*}
As such, Theorem \ref{t:LrPY} quantifies Theorem \ref{t:PY} after uniformly strengthening the assumptions of Theorem \ref{t:PY}. 
\item[iii)] When the functional is strongly stabilizing, this bound takes an even simpler form. More precisely, we say $R_x$ is a stabilization radius at $x$ if 
\begin{align*}
D_x F(\eta|_{B_R(x)}) = D_x F(\eta).
\end{align*}
Then,  applying H\"older's inequality and the uniform moment condition for the add-one-cost leads to the existence a positive finite $c$ such that
\begin{align*}
 \psi_n \le c \sup_{x\in B_n} \P[ R_x\ge b_n ]^{1-\frac{1}{p}}.
\end{align*}
Hence, the upper tail of $R_x$ is relevant in the rate of normal approximation. One may further classify the stabilization condition with regards to the decay of the upper tail. For instance, we say that the funcitonal $F$ is \textbf{exponentially stabilizing} if $R_x$ has a sub-exponential upper tail.
\item[iv)] There are some general methods for obtaining variance lower bounds. For example, one can partition the space into non-overlapping cubes of appropriate size then use projection method for functions of independent random variables such as Heoffding decomposition. Another method via chaos expansion was given in \cite[Section 5]{LPS16}
\end{itemize}
}
\end{Rem}

We mention one application where the second order Poincar\'e estimates do not apply but the two-scale stabilization bounds do. Fix $r>0$ and consider the number $K_n$ of components in the Gilbert graph $G(\eta_n, 2r)$ (or equivalently the Boolean model $O_{r,n} = \cup_{x\in \eta_n} B(x,r)$) as $n\to\infty$. This corresponds to the so-called thermodynamic regime, where the collection of models in the whole space $O_r=\cup_{x\in\eta} B(x,r)$ indexed by $r$ exhibit a phase transition 
at certain threshold $r^*\in (0,\infty)$ e.g. in terms of whether $\P[0 \mbox{ is connected to infinity in } O_r ]=0$.  We stress that the analysis of $K_n$ is relatively involved in the critical phase due to the co-existence of the unbounded occupied component and the unbounded vacant component (in $O_r^c$). However, the following estimate was obtained in \cite{LrPY20+} for all $r>0$  in dimension 2  using the strong stabilization bound
\begin{align*}
d_W((K_n-\E[K_n])/\sqrt{\Var[K_n]},N(0,1))\le \frac{C}{n^{\beta}},
\end{align*}
where $C$ and $\beta$ are  finite positive constants. In $d\ge 3$, a poly-logarithmic rate was  obtained. The bottleneck of these estimates are the two-arm exponents of the critical Boolean models which are hard to improve. 

More generally, when one considers higher dimensional topological statistics of the Boolean model such as the Betti numbers, it may occur that strong stabilization does not hold \cite{YSA17, Tri19, CT20}. In such case, the two-scale weak stabilization bound might be well suited for obtaining quantitative CLT. 
}

\section{Malliavin-Stein for targets in the second Wiener chaos}\label{sec:MS2W}

In this section, we present a short overview on the recent development on Mallaivin-Stein approach for target distributions in the second Gaussian Wiener chaos.  We also formulate some important conjectures that will complement the approach. We adopt the same notation as in Section \ref{ss:isonormal} above. Let $W$ stands for an isonormal Gaussian process on a separable Hilbert space $\HH$. Recall that the elements in the second Wiener chaos are random variables having the general form $F=I_2(f)$, with $f \in  \HH^{\odot 2}$. Notice that, if $f=h\otimes h$, where $h \in \HH$ is such that $\Vert h \Vert_{\HH}=1$, then using the multiplication formula one has $I_2(f) \sim N^2 -1$, where $N \sim \mathscr{N}(0,1)$. To any kernel $f \in \HH^{\odot 2}$, we associate the following \textbf{Hilbert-Schmidt} operator
\begin{equation*}
A_f : \HH \mapsto \HH; \quad g \mapsto f\otimes_1 g. 
\end{equation*} 
We also write $\{\alpha_{f,j}\}_{j \ge 1}$ and $\{e_{f,j}\}_{j \ge 1}$, respectively, to indicate the (not necessarily distinct) eigenvalues of $A_f$ and the corresponding eigenvectors. The next proposition gathers together some relevant properties of the elements of the second Wiener chaos associated with $W$.

\begin{prop}[See Section 2.7.4 in \cite{n-p-book}] \label{second-property}
	Let $F=I_{2}(f)$, $f \in \HH^{ \odot 2}$, be a generic element of the second Wiener chaos of $W$, and write $\{\alpha_{f,k}\}_{k\geq 1}$ for the set of the eigenvalues of the 
	associated Hilbert-Schmidt operator $A_f$.
	
	\begin{enumerate}
		\item The following equality holds: $F=\sum_{k\ge 1} \alpha_{f,k} \big( N^2_k -1 \big)$, where $\{N_k\}_{k \ge 1}$ is a sequence of i.i.d. $\mathscr{N}(0,1)$ random variables that are elements of the isonormal process $W$, and the series converges in $L^2$ and almost surely.
		\item For any $r\ge 2$,
		\begin{equation*}
		\kappa_r(F)= 2^{r-1}(r-1)! \sum_{k \ge 1} \alpha_{f,k}^r.
		\end{equation*}
	\end{enumerate}
\end{prop} 
From now on, for simplicity, we consider the target distributions in the second Wiener chaos of the form 

\begin{equation}\label{eq:form}
F_\infty= \sum_{i=1}^{d} \alpha_{\infty, i} (N^2_i -1)
\end{equation} 
where $N_i \sim \mathscr{N}(0,1)$ are i.i.d, and the coefficients ${\alpha_{\infty, i}}$ are distinct. We also assume that $\E[F^2_\infty]=1$. In the special case when $\alpha_{\infty, i}=1$ for $1 \le i \le d$, the target random variable $F_\infty$ reduces to that of a centered \textbf{chi-squared distribution with $d$ degree of freedom}. The Malliavin-Stein approach has been successfully implemented in a series of papers \cite{n-p-noncentral,d-p,StMethOnWienChaos,InvPrinForHomSums, optimal-gamma}. The target random variables of the form \eqref{eq:form} with $d=2$, and $\alpha_{\infty, 1} \times \alpha_{\infty,2} < 0$ belong to the so-called \textbf{Variance--Gamma} class of probability distributions. We refer to \cite{g-variance-gamma,gaunt-vg-kol,gaunt-VG-perfect-bounds,e-t,a-g} for development of Stein and Malliavin-Stein for the Variance--Gamma distributions. In this setting, the first obstacle for fully developing the Malliavin-Stein approach was the absence of a ``suitable'' Stein operator (meaning by that a differential operator with polynomial coefficients) for the candidate target distribution. This is the message of the next result. Also, the stability phenomenon of the weak convergence of the sequences in the second Wiener chaos is studied in \cite{n-poly} using tools in complex analysis.

\begin{Thm}[Stein characterization \cite{a-a-p-s-stein}]\label{thm:SMC}
Let $F_\infty$ belongs to the second Wiener chaos of the form \eqref{eq:form}. Consider polynomials $Q(x)=\big( P(x)\big)^{2}=\Big(x \prod_{i=1}^{d}(x - \alpha_{\infty, i} ) \Big)^{2}$ and, coefficients \begin{align*}
&  a_l= \frac{P^{(l)}(0)}{l! 2^{l-1}}, \quad 1 \le l \le d+1, \text{ and }\\
&  b_l= \sum_{r=l}^{d+1} \frac{a_r}{(r-l+1)!} \kappa_{r-l+2}(F_\infty).
\quad 2 \le l \le d+1. 
\end{align*}
Assume that $F$ is a general centered random variable living in a
	finite sum of Wiener chaoses (and hence smooth in the sense of
	Malliavin calculus).  Then $F  =  F_\infty$ (equality in
	distribution) if and only if
	$\E \left[ \mathcal{A}_\infty f (F) \right] =0$ for all  polynomials $f:\R \to \R$ where differential operator $\mathcal{A}_\infty$ of order $d$ is
	\begin{equation}\label{eq:SME} 
	\mathcal{A}_\infty f (x):= \sum_{l=2}^{d+1} (b_l - a_{l-1} x )
	f^{(d+2-l)}(x) - a_{d+1} x f(x). 
	\end{equation}
\end{Thm} 
The next essential conjecture formulates the non-Gaussian counterpart of the Stein's Lemma \ref{lem:stein}. An affirmative answer will complete the Stein part of the approach in this delicate setting. 
\begin{con}[Stein Universality Lemma]
Let $\mathcal{H}$ denote an appropriate class of test functions. For every given test function $h \in \mathcal{H}$ consider the associated Stein equation 
\begin{equation}\label{eq:se-2w}
	\mathcal{A}_\infty f (x) = h(x) - \E[h(F_\infty)].
\end{equation}
Then equation \eqref{eq:se-2w} admits a bounded solution $f_h$ which is $d$ times differentiable and that $ \Vert f^{(r)}_h \Vert_\infty < + \infty$ for all $r = 1, \cdots, d$ and the bounds are independent of the test function $h$.
\end{con}

The rest of the section is devoted to the first-ever quantitative estimates with target distributions in the second Wiener chaos. The first estimate is stated in terms of 2-Wasserstein transport distance $\W_2$ (see Section \ref{sec:TD-SD-GC} for definition). We highlight that the upper bound involves only finitely many cumulants, and therefore consistent with one of the ultimate goal of Malliavin-Stein approach.  The second result is more general and rather intricate containing the iterated Gamma operators of the Malliavin calculus.  See also \cite{k} for several related results of a quantitative nature.

\begin{Thm}[\cite{a-a-p-s}]\label{Thm:2w-2w}
Let $F_n=\sum_{k\ge 1} \alpha_{n,k} \big( N^2_k -1 \big)$ be a sequence belongs to the second Wiener chaos associated to the isonormal process $W$ so that $\E[F^2_n]=1$ for all $n\ge 1$. Assume that the target random variable $F_\infty$ as in \eqref{eq:form}. Define 
$$\Delta(F_n) = \sum_{r=2}^{\text{deg}(Q)} \frac{Q^{(r)}(0)}{r!} \frac{\kappa_r(F_n)}{(r-1)!2^{r-1}}.$$
Then there exists a constant $C>0$ depending only on target random variable $F_\infty$ (and independent of $n$) such that 
\begin{equation}\label{eq:2w-2w}
{{\bf \rm W}_2}(F_n,F_\infty) \le \, C \, \bigg( \sqrt{\Delta(F_n)} + \sum_{r=2}^{d+1} \vert \kappa_r(F_n) - \kappa_r(F_\infty) \vert \bigg).
\end{equation}
\end{Thm}

\begin{Ex}{\rm
	Let $d=2$ and $\alpha_{\infty,1}=- \alpha_{\infty,2}=1/2$, then the target
	random variable $F_\infty$ $( = N_1 \times N_2$, where	$N_1,N_2 \sim \mathscr{N}(0,1)$ are independent and equality holds
	in law) belongs to the class of Variance--Gamma
	distributions $VG_c(r,\theta,\sigma)$ with parameters $r=\sigma=1$
	and $\theta=0$. Then, \cite[\rm Corollary 5.10, part (a)]{e-t} reads
	\begin{equation}\label{eq:thale-bound}
	d_W (F_n,F_\infty) \le C\, \sqrt{\Delta(F_n) + 1/4 \, \kappa^2_3(F_n)}
	\end{equation}
which is consistence with estimate \eqref{eq:2w-2w}. One has to note that for the target random variable $F_\infty$ it holds that $\kappa_3(F_\infty)=0$. For a generalization of the estimate \eqref{eq:thale-bound} to the higher moments and $2$- Wasserstein distance see \cite{a-g}. 
}
\end{Ex}

The next result provides a quantitative bound in the \textbf{Kolmogorov distance}. The proof relies on the classical Berry--Essen estimate in terms of bounding the difference of the characteristic functions. We recall that for two real-valued random variables $X$ and $Y$ the Kolmogorov distance is defined as 
$$d_{Kol}(X,Y):= \sup_{x \in \R} \Big \vert  \P(X\in (-\infty, x])    - \P (  Y \in (-\infty, x])\Big \vert.$$

\begin{Thm}[\cite{a-m-p-s}]\label{Thm:Kol-2w}
	Let $F_\infty$ be the target random variable in the second Wiener chaos of the form \eqref{eq:form}. Assume that $\{F_n\}_{ n\ge 1}$ be a sequence of centered random elements living in a finite sum of the Wiener chaoses. Then there exists a constant $C$ (may depend on the sequence $F_n$ but independent of $n$) such that
	
	\begin{equation}\label{eq:Kol-2w}
	\begin{split}
	d_{Kol} (F_n,F_\infty) &\le C \sqrt{  \E \left[ \Big \vert \sum_{r=1}^{d+1} a_r \left( \Gamma_{r-1}(F_n) -  \E[ \Gamma_{r-1}(F_n)]  \right)   \Big \vert \right] + \sum_{r=2}^{d+1} \vert \kappa_r(F_n) - \kappa_r(F_\infty) \vert  }\\
	& \le C \sqrt{ \sqrt{\Var \left(    \sum_{r=1}^{d+1} a_r  \Gamma_{r-1}(F_n)  \right) } +\sum_{r=2}^{d+1} \vert \kappa_r(F_n) - \kappa_r(F_\infty) \vert  }
	\end{split}
	\end{equation} 
\end{Thm}
\begin{Rem}{\rm
	We remark that when the sequence $\{F_n \}_{n\ge 1}$ appearing in Theorem \ref{Thm:Kol-2w} belongs to the second Wiener chaos, then \cite{a-p-p} yields that
	$$ \Var \left(    \sum_{r=1}^{d+1} a_r  \Gamma_{r-1}(F_n) \right) = \Delta(F_n)$$ where the quantity $\Delta(F_n)$ is as Theorem \ref{Thm:2w-2w}. As a result, the estimate \eqref{eq:Kol-2w} takes the form (compare with \eqref{eq:2w-2w}) $$d_{Kol} (F_n,F_\infty) \le C  \sqrt{ \sqrt{ \Delta(F_n)}  +\sum_{r=2}^{d+1} \vert \kappa_r(F_n) - \kappa_r(F_\infty) \vert  }.$$
}
\end{Rem}

We end the section with the following conjecture is aiming to control the iterated Gamma operators of Malliavin calculus appearing in the RHS of the estimate \eqref{eq:Kol-2w} with finitely many cumulants. A successful path might go through first proving the estimate \eqref{eq:Gamma2-conjecture} where we name it as the $\Gamma_2-$ Conjecture. Finally we point out that the estimate \eqref{eq:Gamma2-conjecture} has to be compared with the famous estimate $\Var(\Gamma_1(F)) \le C \kappa_4(F)$ in the normal approximation setting where $F$ is a chaotic random variable.
\begin{con}\label{con:gamma-to-cumulant}
Let $F_\infty$ be the target random variable in the second Wiener chaos of the form \eqref{eq:form}. Assume that $F=I_q(f)$ be a chaotic random variable in the $q$th Wiener chaos with $q \ge 2$. Then there exists a general constant $C$ (may depend on $q$ and $d$) such that 
\begin{equation}\label{eq:gamma-to-cumulant}
	 \Var \left(    \sum_{r=1}^{d+1} a_r  \Gamma_{r-1}(F) \right) \le C \Delta(F).
\end{equation}
In the particular case, when $d=2$, and $\alpha_{\infty,1}=- \alpha_{\infty,2}=1/2$, then the target
random variable $F_\infty$ $( = N_1 \times N_2$, where	$N_1,N_2 \sim \mathscr{N}(0,1)$ are independent, the estimate \eqref{eq:gamma-to-cumulant} boils down to that 
\begin{equation}\label{eq:Gamma2-conjecture}
\Var \left(   \Gamma_{2}(F) - F \right) \le C \left\{  \frac{\kappa_6(F)}{5!} - 2 \frac{\kappa_4(F)}{3!} + \kappa_2(F) \right\}.
\end{equation}
\end{con}

\bibliographystyle{alpha}

\end{document}